\documentclass[a4paper,twoside,english]{amsart}
\usepackage{graphicx}  
\usepackage{amsmath,amssymb,amsthm, amscd}
\usepackage{amssymb,fge}
\usepackage{tikz-cd}
\usepackage[T1]{fontenc}
\usepackage[latin9]{inputenc}
\usepackage{babel}
\usepackage{verbatim}
\makeatletter
\let\@wraptoccontribs\wraptoccontribs
\makeatother

\usepackage{lipsum}

\newcommand\blfootnote[1]{%
  \begingroup
  \renewcommand\thefootnote{}\footnote{#1}%
  \addtocounter{footnote}{-1}%
  \endgroup
}

\usepackage{color}

\usepackage[left=3.2cm,right=3.5cm,top=3cm,bottom=3cm]{geometry}
\usepackage{indentfirst}           
\usepackage{underscore}
\usepackage{enumitem}
\usepackage{relsize}
\usepackage{varwidth}
\usepackage{mathtools} 
\usepackage{hyperref}
\hypersetup{
    colorlinks=true,
    linkcolor=blue,
    filecolor=magenta,      
    urlcolor=cyan,
}
\usepackage{amssymb}
\usepackage[all]{xy}
\makeatletter
\renewcommand*\env@matrix[1][*\c@MaxMatrixCols c]{%
  \hskip -\arraycolsep
  \let\@ifnextchar\new@ifnextchar
  \array{#1}}
\makeatother

\newcommand\Z{\mathbb{Z}}
\newcommand\Q{\mathbb{Q}}

\newcommand\F{{\mathbb{F}}}

\def\F{{\mathbb{F}}}

\def\P{{\mathfrak{P}}}

\def\Q{{\Bbb Q}}

\def\Z{{\Bbb Z}}

\def\P{{\Bbb P}}
\def\C{{\Bbb C}}
\def\F{{\Bbb F}}

\def\CVD{{\hfill\hfil{\lower 2pt\hbox{\vrule\vbox to 7pt
{\hrule width  5pt\varphifill\hrule}\varphirule}}}\par}

\DeclareMathOperator{\Orb}{Orb}

\newcommand{\mysetminus}{\mathbin{\fgebackslash}}

\newtheorem{theorem}{Theorem}[section]
\newtheorem{lemma}[theorem]{Lemma}
\newtheorem{proposition}[theorem]{Proposition}
\newtheorem{proposition-definition}[theorem]{Proposition-Definition}
\newtheorem{corollary}[theorem]{Corollary}
\newtheorem{conjecture}[theorem]{Conjecture}

\theoremstyle{definition}
\newtheorem{definition}[theorem]{Definition}
\theoremstyle{remark}
\newtheorem{remark}{Remark}
\newtheorem{example}{Example}

\theoremstyle{theorem}
\newtheorem{thm}{Theorem}

\theoremstyle{remark}
\newtheorem{rem}{Remark}

\newcommand*{\Scale}[2][4]{\scalebox{#1}{$#2$}}%

\title{Counting points of bounded height in monoid orbits}   
\author[Wade Hindes]{Wade Hindes\\ (\lowercase{with appendix by} Umberto Zannier)}
\begin{document} 
\maketitle
\begin{abstract} Given a set of endomorphisms on $\mathbb{P}^N$, we establish an upper bound on the number of points of bounded height in the associated monoid orbits. Moreover, we give a more refined estimate with an associated lower bound when the monoid is free. Finally, we show that most sets of rational functions in one variable satisfy these more refined bounds.        
\end{abstract} 
\section{Introduction} 
Let $H$ be the absolute multiplicative Weil height on $\mathbb{P}^N(\overline{\mathbb{Q}})$ and let $K$ be a number field. Then given a subset $X\subseteq\mathbb{P}^N(\overline{\mathbb{Q}})$ of interest in some context, the growth rate of the number of $K$-points in $X$ of bounded height, 
\begin{equation*}
X(K,B):=\#\{Q\in X\cap\mathbb{P}^N(K)\,: \, H(Q)\leq B\},
\end{equation*}   
is known to encode interesting invariants of $X$ and $K$. For instance, if $X=\mathbb{P}^N(K)$, then
$X(K,B)\sim C_{K,N}B^{(N+1)[K:\mathbb{Q}]}$ where $C_{K,N}$ depends on the regulator, class group, etc. of $K$. If $X$ is an abelian variety, then $X(K,B)\sim C_{K,X}\log(B)^{r/2}$ where $r$ is the rank of the Mordell-Weil group $X(K)$. If $X$ is a smooth curve of genus at least $2$, then $X(K,B)\sim C_{K,X}$. More generally, if $X$ is a thin set, i.e., a proper Zariski closed subset or the image of some generically finite morphism of degree at least two, then Theorem 3 in \cite[\S13.1]{Serre} implies that 
\begin{equation}\label{serre} X(K,B)\ll B^{(N+1/2)[K:\mathbb{Q}]}\log(B). 
\end{equation} 
Likewise there are a few height-counting results in arithmetic dynamics, where orbits play the role of $X$; see \cite{K3,LRT,KawaguchiSilverman,SilvK3,Markoff} for examples on Markoff varieties, K3 surfaces, and projective space. For instance, suppose that $\phi$ is a dominant rational self-map of $\mathbb{P}^N$ with dynamical degree $\delta_\phi>1$. Then, if $P\in\mathbb{P}^N(K)$ is a point such that the orbit $\Orb_\phi(P)=\{\phi^n(P)\}_{n\geq0}$ is Zariski dense, the Kawaguchi-Silverman Conjecture predicts that \vspace{.05cm} 
\begin{equation}\label{dyndeg}
\#\{Q\in\Orb_\phi(P)\,:\, H(Q)\leq B\}\sim\log(\delta_\phi)^{-1}\log\log(B); \vspace{.05cm} 
\end{equation} 
see \cite{KawaguchiSilverman} for the relevant definitions and background. Of course, this asymptotic is known in the case of morphisms, when $\deg(\phi)=\delta_\phi>1$ and $P$ is not preperiodic. Similarly if $S=\{\phi_1,\dots\phi_s\}$ is a set of endomorphisms of degree at least two equipped with a probability measure $\nu$, then for almost every sequence $\gamma$ of elements of $S$, we have the analogous asymptotic to \eqref{dyndeg} for random orbits: \vspace{.05cm}  
\begin{equation}\label{randomdyndeg} 
\;\;\;\#\{Q\in\Orb_\gamma(P)\,:\, H(Q)\leq B\}\sim \log(\delta_{S,\nu})^{-1}\log\log(B),\qquad\;\;\delta_{S,\nu}=\Pi_{\phi\in S}\deg(\phi)^{\nu(\phi)}.\vspace{.075cm}     
\end{equation}
Here, the bound holds for all $P$ with large enough height; see \cite[Corollary 1.3]{LRT} for details. 

In this paper, we study the problem of counting points of bounded height in monoid (or semigroup) orbits in $\mathbb{P}^N$, that is, counting all of the points of bounded height obtained by applying all possible compositions of maps within a fixed set $S$ to a given initial point $P$; compare to \cite{K3,Markoff}. Intuitively, one expects that if the maps in $S$ are related in some way (for instance, if they commute), then this should cut down the number of possible points in the associated orbits. However, for most $S$ we expect to see no relations (free monoids), and with this in mind, we have the following result; here and throughout, $M_S$ denotes the monoid generated under composition by a set $S$ of endomorphisms of $\mathbb{P}^N$ defined over $\overline{\mathbb{Q}}$. \vspace{.1cm}  

\begin{theorem}{\label{thm:count}} Let $S=\{\phi_1,\dots, \phi_s\}$ be a set of endomorphisms on $\mathbb{P}^N(\overline{\mathbb{Q}})$ with distinct degrees all at least two. If $M_S$ is free, then for all $\epsilon>0$ there exists an effectively computable positive constant $b=b(S,\epsilon)$ and a constant $B_S$ depending only on $S$ such that  
\begin{equation*}
(\log B)^{b}\ll\#\{f\in M_S\,:\, H(f(P))\leq B\}\ll (\log B)^{b+\epsilon}\vspace{.15cm}
\end{equation*} 
holds for all $P\in\mathbb{P}^N(\overline{\mathbb{Q}})$ with $H(P)>B_S$. Moreover, the implicit constants and error terms depend on $P$ and are effectively computable if $B_S$ is.   \vspace{.1cm} 
\end{theorem}
\begin{remark} When $S=\{\phi_1,\phi_2\}$ generates a free monoid with $\deg(\phi_1)=2$ and $\deg(\phi_2)=3$, then we give explicit computations for the bounds in Theorem \ref{thm:count} in Example \ref{eg:twomaps} below.     
\end{remark}  
In particular, we can use the upper bound in Theorem \ref{thm:count} on the number of functions in the free case to give an upper bound on the number of points of bounded height in arbitrary dynamical orbits; compare to \eqref{serre}, to \cite[Theorem 4.15]{K3}, and to the asymptotic for abelian varieties above. In what follows, $\Orb_S(P)=\{f(P):f\in M_S\}$ denotes the \emph{total orbit} of $P$ under the monoid $M_S$. \vspace{.1cm}      
\begin{corollary}{\label{cor:count}} Let $S=\{\phi_1,\dots, \phi_s\}$ be a set of endomorphisms on $\mathbb{P}^N(\overline{\mathbb{Q}})$ all of degree at least two (and distinct if $s\geq2$). Then there exists an effectively computable positive constant $b$ and a constant $B_S$ depending only on $S$ such that  \vspace{.05cm} 
\[\#\{Q\in\Orb_S(P)\,:\, H(Q)\leq B\}\ll (\log B)^{b} \] 
holds for all $P\in\mathbb{P}^N(\overline{\mathbb{Q}})$ with $H(P)>B_S$.  \vspace{.1cm}       
\end{corollary} 
\begin{remark} Although we expect that $\log(B)^{b}$ is also a lower bound for some choice of $b$ and most $S$ (see Conjecture \ref{conjecture} and Theorem \ref{thm:freemonoid} below), we note that it is only an upper bound in general, even for $s\geq2$. For instance, if $M_S$ is a free commutative monoid (e.g., if $S$ is a certain set of monic power maps), then the asymptotic height growth rate in orbits will be a constant times $\log\log(B)$; see \cite[\S 5]{LRT} for details. This matches the case of a single map (also a commutative monoid); see also \eqref{dyndeg} and \eqref{randomdyndeg} above.    
\end{remark} 
Motivated by the upper and lower bounds in Theorem \ref{thm:count}, we conjecture the following exact asymptotic for the number of points (not functions) of bounded height in total orbits associated to free monoids:  \vspace{.1cm} 
\begin{conjecture}\label{conjecture} Let $S=\{\phi_1,\dots, \phi_s\}$ be a set of endomorphisms on $\mathbb{P}^N(\overline{\mathbb{Q}})$ with distinct degrees all at least two. If $M_S$ is free, then there exist constants $a_P=a(S,P)$ and $b=b(S)$ such that \vspace{.05cm} 
\[\lim_{B\rightarrow\infty}\frac{\#\{Q\in\Orb_S(P)\,:\, H(Q)\leq B\}}{(\log B)^{b}}=a_P\vspace{.025cm} \]
holds for all sufficiently generic $P\in\mathbb{P}^N(\overline{\mathbb{Q}})$ (i.e., all $P\in\mathbb{P}^N(\overline{\mathbb{Q}})$ outside of the union of a proper Zariski closed subset and a set of points of bounded height).  
\end{conjecture}  
\begin{remark} Hence, we expect most monoid orbits in $\mathbb{P}^N$ to exhibit similar height growth as: orbits on Markoff varieties \cite{Markoff}, orbits on $K3$ surfaces in $\mathbb{P}^1\times \mathbb{P}^1\times \mathbb{P}^1$ given by $(2,2,2)$-forms \cite[Theorem 4.5]{K3}, and Mordell-Weil groups of abelian varieties. However in these cases, the relevant monoids (or the underlying varieties themselves) form groups, and there is less need to distinguish between counting functions and points. For instance if there are inverses in $M_S$, distinct functions that agree at a point determine a non-trivial fixed point, and these fixed points can typically be controlled. On the other hand in the case of abelian varieties (where one considers the monoid generated by multiplication maps), distinct functions that agree at a point determine a torsion point. Thus this situation may be avoided by throwing away a set of bounded height.  
\end{remark} 
To motivate our conjecture, we restrict our attention to morphisms of $\mathbb{P}^1$. To state our results in this setting, recall that $w\in\mathbb{P}^1(\mathbb{C})$ is called a \emph{critical value} of $\phi\in\mathbb{C}(x)$ if $\phi^{-1}(w)$ contains fewer than $\deg(\phi)$ elements. Likewise, we call a critical value $w$ of $\phi$ \emph{simple} if $\phi^{-1}(w)$ contains exactly $\deg(\phi)-1$ points. In particular, we have the corresponding notions for sets:   

\begin{definition} Let $S=\{\phi_1,\dots,\phi_s\}$ be a set of rational maps on $\mathbb{P}^1$ and let $\mathcal{C}_{\phi_i}$ denote the set of critical values of $\phi_i$. Then $S$ is called \emph{critically separate} if $\mathcal{C}_{\phi_i}\cap\mathcal{C}_{\phi_j}=\varnothing$ for all $i\neq j$. Moreover, $S$ is called \emph{critically simple} if every critical value of every $\phi\in S$ is simple.  
\end{definition} 
As evidence for Conjecture \ref{conjecture} above, we establish the following weak version for generic sets of rational maps on $\mathbb{P}^1$. In particular, we are able to count points instead of just functions.
\begin{theorem}\label{thm:rational} Let $S=\{\phi_1,\dots,\phi_s\}$ be a set of rational maps on $\mathbb{P}^1(\overline{\mathbb{Q}})$ with distinct degrees all at least four. If $S$ is critically separate and critically simple, then $M_S$ is a free monoid and for all $\epsilon>0$ there exists an effectively computable positive constant $b=b(S,\epsilon)$ and a constant $B_S$ depending only on $S$ such that \vspace{.15cm} 
\begin{equation*}
(\log B)^{b}\ll\#\{Q\in \Orb_S(P)\,:\, H(Q)\leq B\}\ll (\log B)^{b+\epsilon}\vspace{.15cm}
\end{equation*} 
holds for all $P\in\mathbb{P}^1(\overline{\mathbb{Q}})$ with $H(P)>B_S$.     
\end{theorem}

Finally, since the theorem above does not directly apply to sets of polynomials, we give a different proof in this case which works quite generically. In what follows, for $m\geq2$ the polynomials $Z_m=x^m$ are called \emph{cyclic} polynomials and the polynomials $T_m$ satisfying $T_m(x+x^{-1})=x^m+x^{-m}$ are called \emph{Chebychev} polynomials (of the first kind).        
\begin{theorem}{\label{thm:freemonoid}} Let $S=\{\phi_1,\dots,\phi_s\}$ be a set of polynomials defined over $\overline{\mathbb{Q}}$, and let $a_ix^{d_i}$ denote the leading term of $\phi_i$. Suppose that $S$ satisfies the following conditions: 
\begin{enumerate}[itemsep=1.5pt,parsep=5pt, topsep=4pt]  
\item[\textup{(1)}] The set of degrees $\{d_1,\dots, d_s\}$ is a multiplicatively independent set in $\mathbb{Z}$.  \vspace{.05cm} 
\item[\textup{(2)}] The set of leading coefficients $\{a_1,\dots, a_s\}$ is a multiplicatively independent set in $\overline{\mathbb{Q}}^{*}$. 
\item[\textup{(3)}] Each $\phi\in S$ is not of the form $F\circ E\circ L$ for some polynomial $F\in\overline{\mathbb{Q}}[x]$, some cyclic or Chebychev polynomial $E$, and some linear $L\in\overline{\mathbb{Q}}[x]$. \vspace{.1cm}     
\end{enumerate}  
Then $M_S$ is a free monoid and for all $\epsilon>0$ there exists an effectively computable positive constant $b=b(S,\epsilon)$ and a constant $B_S$ depending only on $S$ such that \vspace{.15cm} 
\begin{equation*}
(\log B)^{b}\ll\#\{Q\in \Orb_S(P)\,:\, H(Q)\leq B\}\ll (\log B)^{b+\epsilon}\vspace{.15cm}
\end{equation*} 
holds for all $P\in\mathbb{P}^1(\overline{\mathbb{Q}})$ with $H(P)>B_S$.      
\end{theorem}  
We briefly outline the proofs of our results in dimension one above. The first step is to show that $M_S$ is free. For rational functions, this follows from the genus calculations in \cite{Pakovich} and Picard's theorem. For polynomials, conditions (1) and (2) of Theorem \ref{thm:freemonoid} imply that $M_S$ is free; see Theorem \ref{thm:justfree} below. In particular, Theorem \ref{thm:count} implies the desired growth rate on the number of functions $f\in M_S$ with $H(f(P))\leq B$ in both cases. On the other hand, for rational functions the genus calculations in \cite{Pakovich}, Faltings' Theorem, and Tate's telescoping lemma \ref{lem:tate} imply that
\begin{equation}{\label{keylemma}} 
\#\{f\in M_S: f(P)=Q\}
\end{equation} 
is uniformly bounded for all $Q\in\Orb_S(P)$ of sufficiently large height; see Lemma \ref{injective} below. Likewise, the same property holds for polynomials by condition (3) of Theorem \ref{thm:freemonoid} and the integral point classification theorems in \cite{Bilu} and the Appendix \ref{Appendix}. From here, the desired estimate for orbits (for both rational and polynomial functions) follows from Theorem \ref{thm:count} and the uniform bounds on (\ref{keylemma}). We note that it is possible that the full classification theorems in \cite{AZ,Bilu} can be used to strengthen the statement of Theorem \ref{thm:freemonoid}, without reference to leading terms and degrees. However, we have endeavored to give as self-contained and broadly applicable a statement as possible. 
\\[10pt] 
\textbf{Acknowledgements:} We thank Yuri Bilu, Andrew Bridy, Alexander Evetts, Joseph Silverman, and Umberto Zannier for discussions related to this paper. We also thank the authors of \cite{Zieve}; Lemma 3.2 in their paper inspired the proof of Theorem \ref{thm:justfree}. Finally, we are especially grateful to Umberto Zannier (again) for including the appendix to this paper. \vspace{.3cm} 
\section{Auxiliary results} 
To count points of bounded height in orbits, we recall some basic facts about heights and generating functions. However as motivation for what is to come, we begin with a brief sketch of the proof of Theorem \ref{thm:count}, an important ingredient for all other results in this paper. The basic idea, consistent with our earlier work on orbits attached to sequences in \cite{stochastic, Me:dyndeg}, is that the logarithmic height of a point $f(P)\in \Orb_S(P)$ is roughly determined by the size of $\deg(f)$, as long as the initial point $P$ is sufficiently generic; see Lemma \ref{lem:tate} below. With this in mind, to count the number of functions $f\in M_S$ with $\log H(f(P))\leq B$, we should in some sense simply be counting the number of $f$'s of bounded degree. In particular, when $M_S$ is a free monoid, we can relate the number of $f\in M_S$ with bounded degree to the number of restricted integer compositions of bounded size, once we approximate $\log\deg(\phi)$ for all $\phi\in S$ by rational numbers. Finally, we use generating functions (and the location of their poles via Lemma \ref{lem:approx} and Lemma \ref{lem:poles} below) to estimate the number of restricted integer compositions of bounded size. These facts together imply Theorem \ref{thm:count}. With this sketch in place, we move on and review some basic facts about heights.  
\begin{remark} Since multiplicative heights tend to grow exponentially when evaluating functions, it is convenient to use the logarithmic height $h=\log \circ H$ (instead of $H$) to state certain height estimates in dynamics. However, since height-counting on varieties is usually done with multiplicative heights, we convert back to $H$ at the end of the proof of Theorem \ref{thm:count}, to be consistent with similar results in the literature.     
\end{remark} 
Suppose that $\phi:\mathbb{P}^N(\overline{\mathbb{Q}})\rightarrow\mathbb{P}^N(\overline{\mathbb{Q}})$ is is a morphism defined over $\overline{\mathbb{Q}}$ of degree $d_\phi$. Then it is well known that  
\begin{equation}\label{functoriality}
h(\phi(P))=d_\phi h(P)+O_{\phi}(1)\;\;\;\text{for all $P\in\mathbb{P}^N(\overline{\mathbb{Q}})$;} 
\end{equation}  
see, for instance, \cite[Theorem 3.11]{SilvDyn}. With this in mind, we let 
\begin{equation}{\label{htconstant}} 
C(\phi):=\sup_{P \in \mathbb{P}^N(\bar{\mathbb{Q}})} \Big\vert h(\phi(P))-d_\phi h(P)\Big\vert 
\end{equation} 
be the smallest constant needed for the bound in (\ref{functoriality}). Then, in order to control height growth rates when composing arbitrary elements of a set of endomorphisms, we define the following fundamental notion; compare to \cite{stochastic,Me:dyndeg,Kawaguchi}. 
\begin{definition}\label{def:htcontrolled} 
A set $S$ of endomorphisms of $\mathbb{P}^N(\overline{\mathbb{Q}})$ is called \emph{height controlled} if the following properties hold: \vspace{.1cm} 
\begin{enumerate} 
\item $d_S:=\inf\{d_\phi:\phi\in S\}$ is at least $2$. \vspace{.15cm}
\item $C_S:=\sup\{C(\phi): \phi\in S\}$ is finite. \vspace{.1cm}
\end{enumerate} 
\end{definition} 
\begin{remark} We note first that any finite set of morphisms of degree at least $2$ is height controlled. To construct infinite collections, let $T$ be any non-constant set of maps on $\mathbb{P}^1$ and let $S_T=\{\phi\circ x^d\,: \phi\in T,\, d\geq2\}$. Then $S_T$ is height controlled and infinite; a similar construction works for $\mathbb{P}^N$ in any dimension.
\end{remark} 
\begin{remark} Although the results in this paper are for finite $S$, we include the notion of height controlled sets to motivate future work. For instance, many of the tools used below: canonical heights, generating functions, etc. work perfectly well for infinite sets. However, the generating functions that appear in this case are not rational, which adds some subtlety.   
\end{remark} 
As in the case of iterating a single function, it is Tate's telescoping Lemma (generalized below) that allows us to transfer information back and forth between heights and degrees; for a proof, see \cite[Lemma 2.1]{stochastic}.  
\begin{lemma}{\label{lem:tate}} Let $S$ be a height controlled set of endomorphisms of $\mathbb{P}^N(\overline{\mathbb{Q}})$, and let $d_S$ and $C_S$ be the corresponding height controlling constants. Then for all $f\in M_S$, 
\[\bigg|\frac{h(f(Q))}{\deg(f)}-h(Q)\bigg|\leq \frac{C_S}{d_S-1} \;\;\;\; \text{for all $Q\in \mathbb{P}^N(\overline{\mathbb{Q}})$.}\] 
\end{lemma} 
Now that we have a tool to pass from functions yielding a bounded height relation to functions of bounded degree (via Lemma \ref{lem:tate}), we next relate counting functions of bounded degree to counting restricted integer compositions; this is essentially achieved by the fact that $\log\deg(F\circ G)=\log\deg(F)+\log\deg(G)$ for all endomorphisms $F$ and $G$. However, to make this idea precise, we briefly discuss integer compositions, a classical object of study in combinatorics. For more details, see \cite[\S I.3.1]{analytic-combinatorics}.  

Let $T\subseteq\mathbb{N}_{>0}$ be a collection of positive integers (not necessarily finite). Then a \emph{restricted composition} of an integer $n$ with summands in $T$ (or a $T$-restricted composition of $n$) is an \emph{ordered} collection of elements in $T$ whose sum is $n$. For instance, $5=2+3$ and $5=3+2$ are two different restricted compositions of $5$ for the set $T=\{2,3\}$. Given $n$, let $f^T_{n}$ be the number of distinct ways of writing $n$ as a composition with summands (parts) in $T$. Then to give an asymptotic for $f^T_{n}$, one can try and understand the ordinary generating function $f^T(z)=\sum_n f^T_{n}z^n$. In particular, if in addition $f^T(z)$ is a rational or meromorphic function, then the radius of convergence of the generating function, determined by the poles of $f^T(z)$, can be used to deduce an asymptotic for $f^T_{n}$. Luckily, the generating functions for restricted compositions are particularly simple rational functions; see Proposition I.1 in \cite{analytic-combinatorics}.             
\begin{proposition}{\label{prop:OGF}} The ordinary generating function of the number of compositions having summands restricted to a set $T\subseteq \mathbb{N}_{>0}$ is given by 
\[f^T(z)=\frac{1}{1-\sum_{n\in T}z^n}.\]
\end{proposition}
As mentioned above, once we have an expression for $f^T(z)$ as a rational function, we can use the poles of $f^T(z)$ to estimate the $f^{T}_{n}$. Specifically, we have the following Theorem, a simple consequence of partial fractions and Newton expansion. In what follows, if $\mathcal{F}(z)=\sum_n a_nz^n$ is a power series expansion about $z=0$ for a meromorphic function $\mathcal{F}$, then we use the notation $[z^n]\mathcal{F}(z)=a_n$ to extract coefficients. 
\begin{theorem}[Expansion of rational functions]\label{thm:GF} If $\mathcal{F}(z)$ is a rational function that is analytic at zero and has poles at points $\alpha_1, \alpha_2, \dots, \alpha_m$, then its coefficients (as a power series about $0$) are a sum of exponential-polynomials: there exist $m$ polynomials $\{\Pi_j(x)\}_{j=1}^m$ such that for $n$ larger than some fixed $n_0$,  
\[[z^n]\mathcal{F}(z)=\sum_{j=1}^m \Pi_j(n){\alpha_j}^{-n}.\] 
Furthermore, the degree of $\Pi_j$ is equal to the order of the pole of $\mathcal{F}$ at $\alpha_j$ minus one. 
\end{theorem}
In particular, after combining Proposition \ref{prop:OGF} and Theorem \ref{thm:GF}, we see that to obtain an asymptotic formula for the number of integer compositions whose parts are restricted to the  set $\{n_1,\dots,n_s\}$, we must control the roots of smallest modulus of $g(z)=1-(z^{n_1}+\dots+z^{n_s})$. With this in mind, we have the following elementary proposition. 
\begin{lemma}{\label{lem:poles}} Let $n_1, n_2, \dots, n_s$ be positive integers satisfying $\gcd(n_1,n_2,\dots,n_s)=1$. Then the polynomial $g(z)=1-(z^{n_1}+z^{n_2}+\dots+z^{n_s})$ has a unique complex root $\alpha$ of smallest modulus. Moreover, $\alpha$ is the unique positive real root of $g$, and $\alpha$ has multiplicity $one$.    
\end{lemma} 
\begin{proof}  We first show that any positive real root $\alpha$ of $g$ is a root of smallest modulus for $g$ (clearly $g$ has a positive root by the Intermediate Value Theorem). This is a simple consequence of Rouch\'{e}'s Theorem: let $r<\alpha$, let $p(z)=-1-z^{n_1}-\dots-z^{n_s}$, and let $q(z)=2$. Then for all $|z|=r$, we have that  
\begin{equation*}
\begin{split} 
|p(z)|=|-1-z^{n_1}-\dots-z^{n_s}|&\leq 1+|z|^{n_1}+\dots+|z|^{n_s}\\[2pt]
&=1+r^{n_1}+\dots +r^{n_s} \\[2pt] 
&< 1+\alpha^{n_1}+\dots +\alpha^{n_s}=2-(1-\alpha^{n_1}-\dots -\alpha^{n_s})=|q(z)|
\end{split} 
\end{equation*}
by the triangle inequality and since $\alpha$ is a root of $g$. In particular, $p$ and $q$ are holomorphic functions on the disc $D_r$ of radius $r$ such that $|p(z)|<|q(z)|$ on the boundary $D_r$. Hence, Rouch\'{e}'s Theorem implies that $q$ and $q+p=g$ have the same number of roots inside $D_r$. Therefore, $g$ has no complex roots in $D_r$, and $\alpha$ is a root of smallest modulus for $g$. On the other hand, it is clear that $g$ restricted to the positive real numbers is strictly decreasing. Hence, $g$ has only one positive real root. Likewise, it is easy to see that $g'(\alpha)<0$ (since $\alpha$ is positive). Hence, $\alpha$ must be a root of multiplicity one for $g$. 

We next show that $\alpha$ is the unique complex root of $g$ of smallest modulus. This portion of the proof of Lemma \ref{lem:poles} follows from results and arguments in \cite[IV.6]{analytic-combinatorics}, namely the ``Daffodil Lemma" \cite[IV.1]{analytic-combinatorics} and the proof of \cite[Proposition IV.3]{analytic-combinatorics} on the commensurability of dominant directions for rational generating functions arising from regular languages. To see this, suppose that $\zeta=\alpha e^{i\theta}$ is another root of smallest modulus of $g$. Let $f(z)=z^{n_1}+\dots+z^{n_s}$, so that $\zeta$ satisfies $|f(\zeta)|=|1|=1=f(\alpha)=f(|\zeta|)$. In particular, \cite[Lemma IV.1]{analytic-combinatorics} implies that $\theta=2\pi r/p$ for some integers $0\leq r<p$ with $\gcd(r,p)=1$ (when $r\neq0$). Moreover, $f$ admits $p$ as a span; see \cite[Definition IV.5]{analytic-combinatorics}. In particular (since $f$ admits $p$ as a span), $f(z)=z^ah(z^p)$ for some polynomial $h$ and some non-negative integer $a$. Note also that $\gcd(a,p)=1$, since $\gcd(n_1,\dots,n_s)=1$ by assumption. On the other hand, 
\begin{equation*}
\begin{split} 
1=f(\zeta)=\zeta^a\,h(\zeta^p)=(\alpha e^{i2\pi r/p})^a\,h\big((\alpha e^{i2\pi r/p})^p\big)=e^{i2\pi ar/p}\,\alpha^a\,h(\alpha^p)=e^{i2\pi ar/p}\, f(\alpha)=e^{i2\pi ar/p}. 
\end{split}  
\end{equation*}
Hence, $ar/p\in\mathbb{Z}$. But this is impossible unless $r=0$, since $\gcd(ar,p)=1$ otherwise. In particular, $\zeta=\alpha$ and $\alpha$ is the unique complex root of $g$ of smallest modulus as claimed.      
\end{proof} 
Lastly, we include a technical result that allows us to approximate the number of bounded compositions  whose parts are restricted to the set of non-integers $\{\log\deg(\phi_1), \dots,\log\deg(\phi_s)\}$, a task that is equivalent to counting the number of functions in $M_S$ of bounded degree, by integer compositions whose parts satisfy the $\gcd$ condition needed to apply Lemma \ref{lem:poles}.   
\begin{lemma}\label{lem:approx} Let $c_1<c_2< \dots <c_s$ be distinct positive real numbers. Then for all $\delta>0$ there exist positive integers $n_1,\dots, n_s, m_1, \dots, m_s$ and $u$ such that the following conditions hold: \vspace{.1cm} 
\begin{enumerate}
\item[\textup{(1)}] $\displaystyle{c_i-\delta\leq\frac{n_i}{u}< c_i<\frac{m_i}{u}\leq c_i+\delta}$. \vspace{.2cm} 
\item[\textup{(2)}] $\gcd(n_1,\dots, n_s)=1=\gcd(m_1,\dots, m_s)$.   
\end{enumerate}    
\end{lemma}
\begin{remark}\label{rem:extra} In particular, we may assume that $n_1<\dots < n_s<m_1<\dots< m_s$ by choosing $\delta$ sufficiently small.  
\end{remark}  
\begin{proof} Clearly integers $n_1,\dots, n_s, m_1, \dots, m_s$ and $u$ satisfying condition (1) of Lemma \ref{lem:approx} exist. Therefore, to find integers satisfying both (1) and (2), we choose integers satisfying (1) and deform them to ensure that both conditions hold. Specifically, fix an integer $r>0$, let $v=(n_2\dots\cdot n_sm_2\dots\cdot m_s\,u)^r$, and define a new list as follows: \vspace{.05cm}
\begin{equation}\label{eq:newlist}
n_1'= n_1\, v+1,\qquad n_i'= n_i\,v, \qquad m_1'=m_1\,v+1, \qquad m_i'=m_i\, v, \qquad u'=u\, v
\end{equation}
for all $i\neq1$. In particular, we note that $\frac{n_i'}{u'}=\frac{n_i}{u}$ and $\frac{m_i'}{u'}=\frac{m_i}{u}$ for all $i\neq2$ and that  
\[\frac{n_1'}{u'}=\frac{n_1}{u}+\frac{1}{v}\qquad \text{and}\qquad \frac{m_1'}{u'}=\frac{m_1}{u}+\frac{1}{v}.\]
Therefore, we may certainly choose $r$ sufficiently large so that $n_1',\dots, n_s', m_1', \dots, m_s'$ and $u'$ satisfying condition (1), since the original sequence does. On the other hand, it is easy to see that $\gcd(n_1',n_i')=1$ and $\gcd(m_1',m_i')=1$ for all $i\ne2$ by construction. For instance, suppose that $p$ is a prime such that $p|n_1'$ and $p|n_i'$ for some $i\neq2$. Then since $p|n_i'$, we see that $p|v$ or $p|n_i$. But if $p|v$, then $p|n_1v$ and $p|n_1'$. In particular, $p|(n_1'-n_1v)=1$ by \eqref{eq:newlist}, a contradiction. Likewise, if $p|n_i$, then $p|v$ by definition of $v$. Therefore, we may repeat the argument above to reach a contradiction. Similarly, the fact that $\gcd(m_1',m_i')=1$ holds for all $i\ne2$ follows mutatis mutandis. In particular, we see that both conditions (1) and (2) of Lemma \ref{lem:approx} hold for the new list $n_1',\dots, n_s', m_1', \dots, m_s'$ and $u'$, which completes the proof.                   
\end{proof} 
\section{height-counting in orbits}\label{sec:htcount}      
With the necessary background in place, we are ready to prove the bounds on the number of functions $f\in M_S$ yielding a bounded height relation from the Introduction.   
\begin{proof}[(Proof of Theorem \ref{thm:count})] Let $S=\{\phi_1, \dots,\phi_s\}$ be a finite set of endomorphisms on $\mathbb{P}^N$ all of degree at least $2$, and suppose that the monoid $M_S$ generated by $S$ under composition is free. We begin by defining some lengths on $M_S$, which we then relate to integer compositions. Given any vector $\textbf{v}=(v_1,\dots,v_s)\in\mathbb{R}_{>0}^s$ of positive real weights, we define $\mathit{l}_{S,\textbf{v}}(\phi_i)=v_i$ for $\phi_i\in S$ and extend $\mathit{l}_{S,\textbf{v}}$ to all $f\in M_S$ by:  
\begin{equation}\label{lengths} 
\qquad\qquad\mathit{l}_{S,\textbf{v}}(f)=\sum_{j=1}^n \mathit{l}_{S,\textbf{v}}(\theta_j),\;\;\;\text{where $f=\theta_1\circ\theta_2\circ\dots\circ\theta_n$ for some $\theta_j\in S$.}
\end{equation}   
\begin{remark} Note that since $S$ is a free basis of $M_S$ there is a unique way to write $f$ as a composition of elements of $S$. In particular, $\mathit{l}_{S,\textbf{v}}$ is a well-defined function. Alternatively, in the non-free case one can define $\mathit{l}_{S,\textbf{v}}(f)$ by taking an $\inf$ over the possible expressions in (\ref{lengths}).       
\end{remark} 
On the other hand, since $M_S$ is a set of functions there is a natural choice of weighting given by $\textbf{c}=(c_1,\dots,c_s)$ where $c_i=\log\deg(\phi_i)$; moreover, we assume $c_1<c_2<\dots<c_s$. In particular, it follows from the fact that $\deg(F\circ G)=\deg(F)\cdot\deg(G)$ for morphisms that 
\begin{equation}{\label{wts/degrees}}
\mathit{l}_{S,\textbf{c}}(f)=\log\deg(f)\qquad \text{for all $f\in M_S$,} 
\end{equation} 
independent of the generating set. However, non-integer weights (like logs of integers) appear sparingly in the literature, and so we approximate the growth rate of $\mathit{l}_{S,\textbf{c}}$ (which relates to the growth rate of heights in orbits via Tate's telescoping argument) using integer weights. 

To wit, choose positive integers $n_1,\dots, n_s, m_1, \dots, m_s$ and $u$ depending on $\delta$ as in Lemma \ref{lem:approx} and Remark \ref{rem:extra}. Then it follows by construction that $u^{-1}\mathit{l}_{S,\textbf{n}}(f)\leq \mathit{l}_{S,\textbf{c}}(f)\leq u^{-1}\mathit{l}_{S,\textbf{m}}(f)$ for all $f\in M_S$. Hence, \vspace{.05cm}  
\begin{equation}\label{eq:comparison}
\{f\in M_S\,:\, \mathit{l}_{S,\textbf{m}}(f)\leq uB\}\subseteq \{f\in M_S\,:\, \mathit{l}_{S,\textbf{c}}(f)\leq B\}\subseteq\{f\in M_S\,:\, \mathit{l}_{S,\textbf{n}}(f)\leq uB\} \vspace{.05cm}  
\end{equation}  
holds for all positive $B$; here $\textbf{n}=(n_1,\dots,n_s)$ and $\textbf{m}=(m_1,\dots,m_s)$. Now given a positive integer $n$ we define 
\begin{equation}{\label{wt:bd1}}
L_{n}:=\#\{f\in M_S\,:\, \mathit{l}_{S,\textbf{m}}(f)=n\}\qquad\text{and}\qquad U_n:=\#\{f\in M_S\,:\, \mathit{l}_{S,\textbf{n}}(f)=n\}.
\end{equation}       
In particular, since $\textbf{n}$ and $\textbf{m}$ are integer weight vectors, it follows from (\ref{wts/degrees}), (\ref{eq:comparison}) and (\ref{wt:bd1}) that
\begin{equation}{\label{wt:bd2}}
\sum_{n=0}^{[uB]} L_n\leq\#\{f\in M_S\,:\, \log\deg(f)\leq B\}\leq \sum_{n=0}^{[uB]} U_n. 
\end{equation}
Here $[uB]$ denotes the nearest integer to $uB$. On the other hand, since $S$ generates $M_S$ as a free monoid, we can identify $M_S$ with the set of finite sequences of elements of $S$. In particular, $L_n$ (respectively $U_n$) represents the number of ways of writing $n$ as the sum of a sequence of elements in $\{m_1, \dots, m_s\}$ (respectively in $\{n_1,\dots,n_s\}$). Such sequences have been extensively studied in combinatorics \cite[\S I.3.1]{analytic-combinatorics} and are called restricted integer compositions. Specifically, generating functions for these compositions are known; see Proposition \ref{prop:OGF} above. In particular, \vspace{.05cm} 
\begin{equation}{\label{wt:bd3}} 
L_n=\big[z^n\big]\,\frac{1}{1-(z^{m_1}+\dots +z^{m_s})} \qquad \text{and}\qquad U_n=\big[z^n\big]\, \frac{1}{1-(z^{n_1}+\dots +z^{n_s})}. \vspace{.05cm}     
\end{equation}       
As a reminder, $[z^n] \mathcal{F}(z)$ denotes the operation of extracting the coefficient of $z^n$ in the formal power series $\mathcal{F}(z)=\sum f_nz^n$; see \cite[p.19]{analytic-combinatorics}. On the other hand, since $\gcd(n_1,\dots, n_s)=1$ and $\gcd(m_1,\dots, m_s)=1$ by construction, Lemma \ref{lem:poles} implies that both of the rational functions in (\ref{wt:bd3}) have unique poles of smallest modulus (and these poles are positive real numbers of multiplicity one). Let $\alpha_1,\dots\alpha_{r_1}$ be the roots of $g_{\textbf{n}}(z)=1-(z^{n_1}+\dots +z^{n_s})$ arranged in increasing order of modulus and let $\beta_1, \dots,\beta_{r_2}$ be the roots of $g_{\textbf{m}}(z)=1-(z^{m_1}+\dots +z^{m_s})$ arranged in increasing order of modulus. Then Theorem \ref{thm:GF} and (\ref{wt:bd3}) together imply that \vspace{.05cm}  
\begin{equation}\label{wt:bd4} 
\Scale[.91]{L_n=\kappa_1\beta_1^{-n}+p_2(n)\beta_2^{-n}+\dots+p_{r_2}(n)\beta_{r_2}^{-n} \,\;\;\;\;\text{and}\;\;\;\;\, U_n=\tau_1\alpha_1^{-n}+q_2(n)\alpha_2^{-n}+\dots+q_{r_1}(n)\alpha_{r_1}^{-n}} \vspace{.05cm} 
\end{equation} 
for some constants $\kappa_1$ and $\tau_1$ and some polynomials $p_i,q_j\in\mathbb{C}[z]$. Explicitly,
\begin{equation}\label{explicit1}
\kappa_1=\frac{-1}{\beta_1\,g_{\textbf{m}}'(\beta_1)}\quad\text{and}\qquad\tau_1=\frac{-1}{\alpha_1\,g_{\textbf{n}}'(\alpha_1)}. 
\end{equation} 
Here we use the residue method for extracting partial fraction coefficients and Newton's expansion; see the proof of \cite[Theorem IV.9]{analytic-combinatorics}. Moreover, the expressions in (\ref{wt:bd4}) and (\ref{explicit1}) hold simultaneously for all $n>n_0$ for some constant $n_0\in\mathbb{N}$. In particular, by summing (\ref{wt:bd4}) and using the triangle inequality (for both sums and differences) we see that 
\begin{equation}\label{wt:bd5}
\kappa_2\beta_1^{-m}-\kappa_3m^{r_3}|\beta_{2}|^{-m}-\kappa_4\leq\sum_{n=0}^{m} L_n \;\;\;\;\text{and}\;\;\;\;\sum_{n=0}^{m} U_n\leq \tau_2\alpha_1^{-m}+\tau_3m^{r_4}|\alpha_{2}|^{-m}+\tau_4  
\end{equation} 
holds for all $m$ sufficiently large. Again, in the interest of being as explicit as possible (at least for the main terms), we have that \vspace{.1cm}  
\begin{equation}\label{explicit2}
\kappa_2=\frac{\kappa_1(\frac{1}{\beta_1})}{(\frac{1}{\beta_1})-1}=\frac{-1}{\beta_1(1-\beta_1)g_{\textbf{m}}'(\beta_1)}\qquad\text{and}\qquad\tau_2=\frac{\tau_1(\frac{1}{\alpha_1})}{(\frac{1}{\alpha_1})-1}=\frac{-1}{\alpha_1(1-\alpha_1)g_{\textbf{n}}'(\alpha_1)}, \vspace{.1cm} 
\end{equation}
obtained by summing the corresponding geometric series. Moreover, $r_3$ (respectively $r_4$) is the maximum of the multiplicities of the roots of $g_{\textbf{m}}$ (respectively $g_{\textbf{n}}$) minus one. Hence, after taking $m=[Bu]$, combining (\ref{wt:bd2}) and (\ref{wt:bd5}), and absorbing $u$ into the relevant constants, we see that \vspace{.1cm} 
\begin{equation}\label{wt:bd6} 
\kappa_5\,C_1^B-\kappa_6\,B^{r_3}\,C_{2}^B-\kappa_4\leq\#\{f\in M_S\,:\, \log\deg(f)\leq B\}\leq\tau_5\,C_{3}^{B}+\tau_6\,B^{r_4}C_{4}^{B}+\tau_4 \vspace{.1cm} 
\end{equation} 
holds for all $B$ sufficiently large; here we use also that $Bu-1\leq[Bu]\leq Bu+1$, so that (some) of the relevant constants are given explicitly by \vspace{.1cm}   
\begin{equation}\label{explicit3}
\begin{split} 
&C_1=\frac{1}{\beta_1^u}, \qquad \kappa_5=\kappa_2\beta_1=\frac{1}{(\beta_1-1)\,g_{\textbf{m}}'(\beta_1)}, \qquad\; C_2=\frac{1}{\;|\beta_2|^u},\\[10pt]
&C_3=\frac{1}{\alpha_1^u}, \qquad \tau_5=\frac{\tau_2}{\alpha_1}=\frac{1}{\alpha_1^2\,(\alpha_1-1)\,g_{\textbf{n}}'(\alpha_1)},\qquad C_4=\frac{1}{\;|\alpha_2|^u}.\\[5pt]  
\end{split} 
\end{equation} 
We note in particular that $C_1>C_2$ and $C_3>C_4$, since $\beta_1<|\beta_2|$ and $\alpha_1<|\alpha_2|$ by construction. Now suppose that $P\in\mathbb{P}^N(\overline{\mathbb{Q}})$ is such that $h(P)>b_S:=C_S/(d_S-1)$, where $C_S$ and $d_S$ are the constants from Definition \ref{def:htcontrolled} above. Then, Tate's telescoping Lemma \ref{lem:tate} implies that \vspace{.05cm}
\[\deg(f)(h(P)-b_S)\leq h(f(P))\leq\deg(f)(h(P)+b_S).\vspace{.05cm}\] 
Therefore, for all $B$ we have the subset relations: \vspace{.1cm} 
\begin{equation}\label{wt:bd7}
\Scale[.9]{\begin{split} 
\bigg\{f\in M_S\,:\,\log\deg(f)\leq \log\bigg(\frac{B}{h(P)+B_S}\bigg)\bigg\}&\subseteq\big\{f\in M_S:\, h(f(P))\leq B\big\} \\[5pt] 
&\subseteq\bigg\{f\in M_S\,:\,\log\deg(f)\leq \log\bigg(\frac{B}{h(P)-B_S}\bigg)\bigg\}.\\[5pt]  
\end{split} }
\end{equation} 
In particular, if we replace $B$ with $\log(B/(h(P)+B_S))$ on the left side of (\ref{wt:bd6}), replace $B$ with $\log(B/(h(P)-B_S))$ on the right side of (\ref{wt:bd6}), and apply the change of base formulas for logarithms, then we deduce from (\ref{wt:bd6}) and (\ref{wt:bd7}) that \vspace{.1cm} 
\begin{equation}\label{wt:bd8} 
\begin{split}
&\Scale[1.15]{\;\;\;\;\;\kappa_5 \bigg(\frac{B}{h(P)+b_S}\bigg)^{\log(C_1)}-\;\kappa_6\,\log\bigg(\frac{B}{h(P)+b_S}\bigg)^{r_3}\bigg(\frac{B}{h(P)+b_S}\bigg)^{\log(C_2)}-\;\kappa_4}\\[15pt] 
&\leq\;\#\big\{f\in M_S:\, h(f(P))\leq B\big\} \\[15pt] 
&\leq\;\Scale[1.15]{\tau_5 \bigg(\frac{B}{h(P)-b_S}\bigg)^{\log(C_3)}+\;\tau_6\,\log\bigg(\frac{B}{h(P)-b_S}\bigg)^{r_4}\bigg(\frac{B}{h(P)-b_S}\bigg)^{\log(C_4)}+\,\;\tau_4} \\[6pt] 
\end{split}
\end{equation} 
holds for all $B$ sufficiently large and all initial points $P$ such that $h(P)>b_S$. Moreover, since most height counting problems on varieties are stated in terms of multiplicative heights, we replace $B$ with $\log B$ in (\ref{wt:bd8}) to obtain \vspace{.1cm} 
\begin{equation}\label{wt:bd9} 
\begin{split}
&\Scale[.9]{\bigg(\frac{\kappa_5}{(h(P)+b_S)^{\log(C_1)}}\bigg) \log(B)^{\log(C_1)}-\;\bigg(\frac{\kappa_6}{(h(P)+b_S)^{\log(C_2)}}\bigg)\log\bigg(\frac{\log B}{h(P)+b_S}\bigg)^{r_3}\log(B)^{\log(C_2)}-\;\kappa_4}\\[13pt] 
\leq&\;\#\big\{f\in M_S:\, H(f(P))\leq B\big\} \\[13pt] 
\leq&\;\Scale[.9]{\bigg(\frac{\tau_5}{(h(P)-b_S)^{\log(C_3)}}\bigg) \log(B)^{\log(C_3)}+\;\bigg(\frac{\tau_6}{(h(P)-b_S)^{\log(C_4)}}\bigg)\log\bigg(\frac{\log B}{h(P)-b_S}\bigg)^{r_4}\log(B)^{\log(C_4)}+\;\tau_4}\,. \\[5pt] 
\end{split} 
\end{equation}
Hence, after renaming the constants above, we see that there exist positive constants $a_{1}(S,P,\delta)$, $a_{2}(S,P,\delta)$, $b_1(S,\delta)$, $b_{2}(S,\delta)$ and $B_S:=e^{b_S}$ such that \vspace{.15cm}   
\begin{equation}\label{key}
\;a_1(\log B)^{b_1}+o\big((\log B)^{b_1}\big)\leq\#\{f\in M_S\,:\, H(f(P))\leq B\}\leq a_2(\log B)^{b_2}+o\big((\log B)^{b_2}\big)\; \vspace{.15cm}
\end{equation} 
holds for all $P\in\mathbb{P}^N(\overline{\mathbb{Q}})$ with $H(P)\geq B_S$. Moreover, $b_{1}$ and $b_{2}$ depend only on the set $S$ and $\delta$, and $a_{1}$ and $a_{2}$ (and the lower order terms) depend on $S$, $\delta$ and $P$. Specifically, (\ref{explicit1}), (\ref{explicit2}) and (\ref{explicit3}) together imply 
\begin{equation}\label{explicit4} 
\begin{split} 
&a_1=\frac{1}{(\beta_1-1)\,g_{\textbf{m}}'(\beta_1)\,\log\big(B_S\,H(P)\big)^{\log(\beta_1^{-u})}},\qquad b_1=\log(\beta_1^{-u}),\\[10pt] 
&a_2=\frac{1}{\alpha_1^2\,(\alpha_1-1)\,g_{\textbf{n}}'(\alpha_1)\,\log\Big(\frac{H(P)}{B_S}\Big)^{\log(\alpha_1^{-u})}}, \qquad\;\; b_2=\log(\alpha_1^{-u}). 
\end{split} 
\end{equation} 
Moreover, since roots of polynomials can be approximated to any accuracy effectively, $b_1$ and $b_2$ can be computed effectively (also integers as in Lemma \ref{lem:approx} can be produced effectively for all $\delta$). Therefore, to complete the proof of Theorem \ref{thm:count}, we need only show that the difference $b_2-b_1>0$ can be made arbitrarily small (by letting $\delta$ go to zero); see \eqref{smallF} below. Then we set $b=b_1$ and $b_2=b_1+\epsilon$ to deduce the claim in Theorem \ref{thm:count}. 

To do this, we use the Mean Value Theorem applied to the functions $f(x)=-g_{\textbf{m}}(x)$ and $\mathfrak{h}(x)=u\log(x)$ on the intervals $[\alpha_1,\beta_1]$. With this in mind, we begin with a few estimates, all of which follow easily from part (1) of Lemma \ref{lem:approx}:  
\begin{equation}\label{smallA}
\frac{2\delta}{c_1}<\frac{2\delta u}{n_1}<\frac{2\delta}{c_1-\delta},\qquad 1<\frac{m_1}{n_1}<\frac{c_1+\delta}{c_1-\delta},\qquad \frac{u}{m_1}<\frac{1}{c_1}.
\end{equation} 
To simplify the expressions that follow, let $\alpha=\alpha_1$ and $\beta=\beta_1$. Then since $n_1\leq n_i$ and $0<\alpha<1$, we see that $1=\sum_{i=1}^s\alpha^{n_i}\leq s\alpha^{n_1}$. Therefore, 
\begin{equation}\label{smallB}
\Big(\frac{1}{s}\Big)^{\frac{1}{n_1}}\leq\alpha. 
\end{equation}
In particular, \eqref{smallA} and \eqref{smallB} together imply the following lower bound on the derivative:
\begin{equation}\label{smallC}
\;f'(\alpha)=m_s\alpha^{m_s-1}+\dots+m_1\alpha^{m_1-1}\geq m_1\alpha^{m_1-1}\geq m_1\alpha^{m_1}\geq m_1 \Big(\frac{1}{s}\Big)^{\frac{m_1}{n_1}}\geq m_1\Big(\frac{1}{s}\Big)^{\frac{c_1+\delta}{c_1-\delta}}. 
\end{equation}                      
Similarly, \eqref{smallA} and \eqref{smallB} together imply that:
\begin{equation*}
\begin{split} 
f(\alpha)&=\alpha^{(\frac{m_s}{u}-\frac{n_s}{u})u}\cdot\alpha^{n_s}+\dots+ \alpha^{(\frac{m_1}{u}-\frac{n_1}{u})u}\cdot\alpha^{n_1}-1\\[3pt]
&\geq \alpha^{2\delta u}\cdot\alpha^{n_s}+\dots+\alpha^{2\delta u}\cdot\alpha^{n_1}-1\\[3pt]
&=\alpha^{2\delta u}(\alpha^{n_s}+\dots+\alpha^{n_1})-1\\[3pt]
&=\alpha^{2\delta u}-1\geq \Big(\frac{1}{s}\Big)^{\frac{2\delta u}{n_1}}-1\geq \Big(\frac{1}{s}\Big)^{\frac{2\delta}{c_1-\delta}}-1.  
\end{split} 
\end{equation*} 
Here, we use also that $0\leq \frac{m_i}{u}-\frac{n_i}{u}\leq 2\delta$ by construction; see Lemma \ref{lem:approx} part (1). In particular, we deduce the following key upper bound:  
\begin{equation}\label{smallD}
-f(\alpha)\leq 1-\Big(\frac{1}{s}\Big)^{\frac{2\delta}{c_1-\delta}}. 
\end{equation}
We are now ready to apply the Mean Value Theorem to $f(x)$ on $[\alpha,\beta]$. Specifically, 
\begin{equation*}
 m_1\Big(\frac{1}{s}\Big)^{\frac{c_1+\delta}{c_1-\delta}}\leq f'(\alpha)=\min_{\alpha\leq x\leq \beta}f'(x)\leq \frac{f(\beta)-f(\alpha)}{\beta-\alpha}=\frac{-f(\alpha)}{\beta-\alpha}\leq\frac{1-(\frac{1}{s})^{\frac{2\delta}{c_1-\delta}}}{\beta-\alpha}
\end{equation*}
follows from \eqref{smallC}, \eqref{smallD}, and the Mean Value Theorem. Therefore, we have the estimate:  
\begin{equation}\label{smallE}
0\leq \beta-\alpha\leq \frac{1-(\frac{1}{s})^{\frac{2\delta}{c_1-\delta}}}{m_1(\frac{1}{s})^{\frac{c_1+\delta}{c_1-\delta}}}.
\end{equation} 
Likewise, the Mean Value Theorem for $\mathfrak{h}(x)=u\log(x)$ on $[\alpha,\beta]$, \eqref{smallB}, and the fact that $n_1>0$ together yield 
\begin{equation}\label{MVTLog}
0\leq\frac{\mathfrak{h}(\beta)-\mathfrak{h}(\alpha)}{\beta-\alpha}\leq\max_{\alpha\leq x\leq\beta}\mathfrak{h}'(x)=\mathfrak{h}'(\alpha)=u\alpha^{-1}\leq su.
\end{equation} 
Hence, after combining \eqref{explicit4},\eqref{smallA}, \eqref{smallE} and \eqref{MVTLog}, we deduce that
\begin{equation}\label{smallF}
\begin{split}
\Scale[1.05]{
0\leq b_2-b_1=\mathfrak{h}(\beta)-\mathfrak{h}(\alpha)\leq su\cdot \frac{1-(\frac{1}{s})^{\frac{2\delta}{c_1-\delta}}}{m_1(\frac{1}{s})^{\frac{c_1+\delta}{c_1-\delta}}}= s\cdot\frac{u}{m_1}\cdot \frac{1-(\frac{1}{s})^{\frac{2\delta}{c_1-\delta}}}{(\frac{1}{s})^{\frac{c_1+\delta}{c_1-\delta}}}\leq\frac{s}{c_1}\cdot\frac{1-(\frac{1}{s})^{\frac{2\delta}{c_1-\delta}}}{(\frac{1}{s})^{\frac{c_1+\delta}{c_1-\delta}}}} 
\end{split} 
\end{equation} 
However, the upper bound in \eqref{smallF} goes to zero as $\delta$ goes to zero. Therefore, the exponents $b_1$ and $b_2$ in \eqref{key} can be made arbitrarily close.             
\end{proof} 
\begin{remark} If $S$ has only two maps ($s=2$), then the trinomials $g_{\textbf{n}}(z)=1-z^{n_1}-z^{n_2}$ and $g_{\textbf{m}}(z)=1-z^{m_1}-z^{m_2}$ must have non-zero discriminant (in fact, here we need only that $n_1\neq n_2$ and $m_1\neq m_2$, making no assumptions on gcd's); this fact follows easily from the discriminant formula in \cite[Theorem 4]{trinomial}. In particular, $r_3$ and $r_4$ from (\ref{wt:bd6}) and (\ref{wt:bd9}) must be zero. Hence, we obtain simpler bounds for the number of functions of bounded degree (hence, also for the number of points of bounded height in orbits). For instance,
\[\kappa_5\,C_1^B-\kappa_6\,C_{2}^B-\kappa_4\leq\#\{f\in M_S\,:\, \log\deg(f)\leq B\}\leq\tau_5\,C_{3}^{B}+\tau_6\,C_{4}^{B}+\tau_4\]
holds for all $B$ sufficiently large.    
\end{remark} 
\begin{example}{\label{eg:twomaps}} In particular, if $S=\{\phi_1, \phi_2\}$ with $\deg(\phi_1)=2$ and $\deg(\phi_2)=3$, then we use the crude approximations
\[\frac{79}{115}<\log(2)<\frac{80}{115}\qquad\text{and}\qquad \frac{126}{115}<\log(3)<\frac{127}{115}\]
as inputs to Lemma \ref{lem:approx} to obtain some explicit bounds for Theorem \ref{thm:count}. Specifically,  
\begin{align*} 
\Scale[.83]{\Bigg(\frac{1.46457}{\log(B_S\,H(P))^{0.78437}}\Bigg)\log(B)^{0.78437}+o\big(\log(B)^{0.78437}\big)}&\leq\#\{f\in M_S\,:\, H(f(P))\leq B\}\\
&\leq \Scale[.85]{\Bigg(\frac{1.48541}{\log\big(\frac{H(P)}{B_S}\big)^{0.79232}}\Bigg)\log(B)^{0.79232} + o\big(\log(B)^{0.79232}\big)}
\end{align*} 
holds for all $P\in\mathbb{P}^N(\overline{\mathbb{Q}})$ of sufficiently large height; here we use (\ref{explicit3}), (\ref{wt:bd9}) and {\tt{Magma}} \cite{Magma} to approximate roots of polynomials.   
\end{example} 
Lastly, we can use the bounds in Theorem \ref{thm:count} on the number of functions in free monoids satisfying a bounded height relation to give an upper bound on the number of points of bounded height in arbitrary monoid orbits. 
\begin{proof}[(Proof of Corollary \ref{cor:count})] Let $S=\{\phi_1,\dots,\phi_s\}$ be a set of endomorphisms all of degree at least $2$. If $s=1$ (i.e., $S=\{\phi\}$ contains just one map), then one may use the canonical height \cite[\S3.4]{SilvDyn} associated to $\phi$ to reach the desired bound. Namely, the fact that $|\hat{h}_\phi-h|\leq c_\phi$ and that $\hat{h}(\phi^n(P))=d_\phi^n\,\hat{h}_\phi(P)$ together imply that 
\[
\Scale[.95]{\Bigg\{n\,:\, n\leq \log_{d_\phi}\bigg(\frac{\log(B)-c_\phi}{\hat{h}_\phi(P)}\bigg)\Bigg\}\subseteq\{Q\in\Orb_\phi(P): H(Q)\leq B\}\subseteq\Bigg\{n\,:\, n\leq \log_{d_\phi}\bigg(\frac{\log(B)+c_\phi}{\hat{h}_\phi(P)}\bigg)\Bigg\} } 
\] 
for all non-preperiodic $P$. On the other hand, if $P$ is preperiodic, then $\Orb_\phi(P)$ is finite. In particular, the number of points with (multiplicative) height at most $B$ is certainly bounded above by a constant times $\log\log(B)\ll\log(B)$ as claimed; hence, $b=1$ in this case.  

Now assume that $s\geq2$, and let $F_S$ be the free monoid generated by $S$ under concatenation. Then, given a word $w=\theta_1\dots\theta_n\in F_S$, we can define an action of $w$ on $\mathbb{P}^N(\overline{\mathbb{Q}})$ via $w\cdot P= \theta_1\circ\dots\circ\theta_n(P)$. Likewise, we define the degree of $w$ to be $\deg(\theta_1\circ\dots\circ\theta_n)$. In particular, (by counting words of bounded degree) it is straightforward to see that we can replace $M_S$ with $F_S$ in the proof of Theorem \ref{thm:count} and deduce that  \vspace{.075cm}
\[a_1\log(B)^{b_1}+o\big(\log(B)^{b_1}\big)\leq \#\{w\in F_S\,:\, H(w\cdot P)\leq B\}\leq a_2\log(B)^{b_2}+o\big(\log(B)^{b_2}\big) \]  
for some constants $a_1(P),a_2(P),$ $b_1$ and $b_2$ (whenever $H(P)> B_S,$ as before); here we can choose $\delta=0.1$, small enough to separate logs of distinct integers (see Remark \ref{rem:extra}). In particular, since every point $Q\in\Orb_S(P)$ is of the form $Q=w\cdot P$ for some $w\in F_S$, we have that 
\[\#\{Q\in\Orb_S(P)\,:\, H(Q)\leq B\}\leq\#\{w\in F_S\,:\, H(w\cdot P)\leq B\}\leq a_2\log(B)^{b_2}+o\big(\log(B)^{b_2}\big).\]   
Therefore, the number of points in $\Orb_S(P)$ with height at most $B$ is $\ll\log(B)^{b_2}$. Concretely, by choosing $\delta={0.1}$ we get the crude bound $b_2\leq\frac{1}{(c_1-0.1)}\log(s)$ from \eqref{explicit4} and \eqref{smallB}.         
\end{proof} 
\begin{remark} It is likely that the statement and proof of Theorem \ref{thm:count} hold for height controlled sets of simultaneously polarizable maps on \emph{any} projective variety. The main arithmetic ingredient, Tate's telescoping Lemma \ref{lem:tate}, works perfectly well with this level of generality; see \cite[Lemma 2.1]{stochastic}. Moreover, the other components of the proof (generating functions and diophantine approximation of degrees) don't depend on $\mathbb{P}^N$.    
\end{remark}  
\section{Monoid orbits in dimension one}  
In this section, we prove Theorems \ref{thm:rational} and \ref{thm:freemonoid} on monoid orbits over $\mathbb{P}^1$. To do this, we first show that the relevant sets of maps generate free monoids under composition. For critically separate and simple sets of rational maps, this follows directly from the main results of \cite{Pakovich}.  
\begin{theorem}\label{thm:freerational} Let $S=\{\phi_1,\dots,\phi_s\}$ be a set of rational maps on $\mathbb{P}^1(\mathbb{C})$ all of degree at least four. If $S$ is critically separate and critically simple, then $M_S$ is a free. 
\end{theorem}  
\begin{proof} Suppose that $f_1=\theta_1\circ\dots\circ\theta_n=\tau_1\circ\dots\circ\tau_m=g_1$ for some $\theta_i,\tau_j\in S$. Without loss of generality, we may assume that $n\geq m$. Clearly if $n=m=1$, then $\theta_1=\tau_1$ and there is nothing to prove. Therefore, we may assume that $n\geq m>1$. Write $f_2=\theta_2\circ\dots\circ\theta_n$ and $g_2=\tau_2\circ\dots\circ\tau_m$ so that $\theta_1(f_2)=\tau_1(g_2)$. However, since $f_2$ and $g_2$ are non-constant and $S$ is critically separate, \cite[Theorem 1.1]{Pakovich} implies that $\theta_1=\tau_1$. Likewise since $S$ is critically simple and $\deg(\theta_1)\geq4$, we see that $f_2=g_2$ by \cite[Theorem 1.3]{Pakovich}. Repeating the argument above now for $f_2$ and $g_2$ (instead of $f_1$ and $f_2$), we see that $\theta_2=\tau_2$ and $\theta_3\circ\dots\circ\theta_n=\tau_3\circ\dots\circ\tau_m$. We can clearly keep going to deduce that $\theta_i=\tau_i$ for all $1\leq i\leq m$. Finally, by equating degrees given by the original relation $f_1=g_1$, we see that $\deg(\theta_{n-m})\cdots\deg(\theta_n)=1$, a contraction unless $n=m$. This completes the proof that $M_S$ is free.          
\end{proof} 
Next we show that polynomial sets with multiplicatively independent degrees and leading coefficients generate free monoids under composition. This is perhaps known to the experts. However, without a reference, we include a proof for completeness. Our argument is inspired by the proof of \cite[Lemma 3.2]{Zieve}.  
\begin{theorem}{\label{thm:justfree}} Let $S=\{\phi_1,\dots,\phi_s\}$ be a set of polynomials defined over a field $K$ of characteristic zero, and let $a_ix^{d_i}$ denote the leading term of $\phi_i$. If $\{d_1,\dots, d_s\}$ is a multiplicatively independent set in $\mathbb{Z}$ and $\{a_1,\dots, a_s\}$ is a multiplicatively independent set in $K^{*}$, then $M_S$ is a free monoid.      
\end{theorem} 
\begin{proof} As the statement of the theorem suggests, it suffices to study the monoid generated by the leading terms in $S$. To make this statement precise, we note the following lemma: 
\begin{lemma}{\label{lem:leadingterms}} Let $S=\{\phi_1,\dots,\phi_s\}$ be a set of polynomials defined over a field $K$, let $a_ix^{d_i}$ denote the leading term of $\phi_i$, and let $S'=\{a_1x^{d_1},\dots,a_sx^{d_s}\}$. If $M_{S'}$ is a free monoid, then $M_S$ is a free monoid.   
\end{lemma}  
\begin{proof} This statement is a simple consequence of the fact that $\mathit{lt}(f\circ g)=\mathit{lt}(f)\circ\mathit{lt}(g)$ for all $f,g\in K[x]$; here $\mathit{lt}(\cdot)$ denotes the leading term of a polynomial. To see this, suppose that $M_{S'}$ is a free monoid and that there is some relation 
\begin{equation}\label{relation1} 
\theta_1\circ\theta_2\circ\dots\circ\theta_n=\tau_1\circ\tau_2\circ\dots\circ\tau_m
\end{equation} 
for some $\theta_i,\tau_j\in S$. Then, in particular, we have an equality of leading terms,
\[\mathit{lt}(\theta_1)\circ\mathit{lt}(\theta_2)\circ\dots\circ\mathit{lt}(\theta_n)=\mathit{lt}(\tau_1)\circ\mathit{lt}(\tau_2)\circ\dots\circ\mathit{lt}(\tau_m).\]
But this is a relation in $M_{S'}$, which is free on the letters in $S'$. Therefore, $n=m$ and $\mathit{lt}(\theta_i)=\mathit{lt}(\tau_i)$. However, again since $M_{S'}$ is free, $\mathit{lt}(\theta_i)=\mathit{lt}(\tau_i)$ implies that $\theta_i=\tau_i$. Hence the relation in \eqref{relation1} is a trivial one.        
\end{proof} 
Now back to the proof of Theorem \ref{thm:justfree}. In particular, in light of Lemma \ref{lem:leadingterms}, we may assume that $S=\{\phi_1,\dots\phi_s\}$ is a set of monomials with $\phi_i=a_ix^{d_i}$, that $\{d_1,\dots d_s\}$ is a multiplicatively independent set in $\mathbb{Z}$, and that $\{a_1,\dots a_s\}$ is a multiplicatively independent set in $K^{*}$. Now, given $F=\theta_1\circ\dots\circ\theta_n\in M_S$ and $\phi\in S$, we define $e_\phi(F)=\#\{j\,|\,\theta_j=\phi\}$ to be the number of $\phi$'s appearing in the string defining $F$ (strictly speaking this is an abuse of notation; $e_\phi$ is a function on words). In particular, if there is a relation $F=G$ for some $F,G\in M_S$, then we see that 
\begin{equation}\label{eq:degrees}
d_1^{\,e_{\phi_1}(F)}\cdots d_s^{\,e_{\phi_s}(F)}=\deg(F)=\deg(G)=d_1^{\,e_{\phi_1}(G)}\cdots d_s^{\,e_{\phi_s}(G)}.
\end{equation} 
However, the $d_i$'s are multiplicatively independent by assumption, so that $e_{\phi_i}(F)=e_{\phi_i}(G)$ for all $i$. In particular, the strings defining $F$ and $G$ have the same length (i.e., the total number of letters from $S$) equal to $n=\sum e_{\phi_i}(F)$. Hence, 
\begin{equation}\label{relation2}
\qquad F=\theta_1\circ\dots\circ\theta_n=\tau_1\circ\dots\circ \tau_n=G\qquad\text{for some}\; \theta_i, \tau_i\in S.
\end{equation} 
Moreover, $e_{\phi_i}(F)=e_{\phi_i}(G)$ for all $i$. From here, we will show that $\theta_i=\tau_i$ by induction on the length $n$. The $n=1$ case is clear. For $n>1$, if \eqref{relation2} holds then 
\begin{equation}\label{relation3}
F'\circ\theta=F=G=G'\circ\tau
\end{equation} 
for some $\theta,\tau\in S$ and some monomials $F'$ and $G'$ given by strings of length $n-1$ of elements of $S$. We proceed in cases. \\[3pt]
\textbf{Case(1):} Suppose that $\theta=\tau$, and write $\theta=ax^d$, $F'=a_{F'}\,x^{\deg(F')}$ and $G'=a_{G'}\,x^{\deg(F')}$. Here we use that $\deg(F)=\deg(G)$ and $\theta=\tau$, so that $\deg(F')=\deg(G')$. Therefore, \eqref{relation3} becomes
\[a_{F'}\,a^{\deg(F')}\,x^{d\deg(F')}=a_{G'}\,a^{\deg(F')}\,x^{d\deg(F')},\] 
and we deduce that $a_{F'}=a_{G'}$. However, then $F'=a_{F'}\,x^{\deg(F')}=a_{G'}\,x^{\deg(F')}=G'$ and $F',G'\in M_S$ are polynomials obtained by composing strings of elements of $S$ of length $n-1$. In particular, we may deduce that $\theta_i=\tau_i$ for all $i<n$ by induction. On the other hand, $\theta_n=\theta=\tau=\tau_n$ by construction. Therefore, $\theta_i=\tau_i$ for all $i\leq n$ as claimed.  \\[3pt] 
\textbf{Case(2):} Suppose that $\theta\neq\tau$. We fix some notation. Given a string $\theta_1\dots \theta_m$ of elements of $S$, write 
\[f=\theta_1\circ\dots\circ\theta_m=a_f\,x^{\deg(f)}=(a_1^{n_1}\cdots a_s^{n_s}) \,x^{\deg(f)}.\]
Then define the $a_i$-degree of $f$ (or more accurately, the $a_i$-degree of the corresponding string) to be $\deg_{a_i}(f)=n_i$.  Note that this construction is well-defined since the leading coefficient $a_f$ is in the (multiplicative) semigroup generated by the $a_i$'s and the $a_i$'s are multiplicatively independent by assumption. Now write $\theta=ax^d$. Then we will show that $\deg_a(F)\neq\deg_a(G)$, a contradiction, using \eqref{relation3}, the fact that $\theta\neq\tau$, and the following elementary observations about $a$-degrees: 
\begin{lemma}\label{lem:a-degree} Let $S$ be as in Theorem \ref{thm:justfree} and let $\theta=ax^d\in S$. Then the following statements hold: 
\begin{enumerate} 
\item[\textup{(1)}] If $f_1,f_2,g\in M_S$, $\deg_a(f_1)\leq\deg_a(f_2)$, and $\deg(f_1)\leq\deg(f_2)$, then $\deg_a(f_1\circ g)\leq\deg_a(f_2\circ g)$. \\[2pt] 
\item[\textup{(2)}] Let $f\in M_S$, and suppose that $e_\theta(f)=e\geq1$. Then $\deg_a(f)\leq \frac {d^e-1}{d-1}\cdot\frac{\deg(f)}{d^e}$. \vspace{.1cm} 
\end{enumerate}  	
\end{lemma} 

We grant Lemma \ref{lem:a-degree} for now and return to the proof later. To see that $\deg_a(F)\neq\deg_a(G)$ in Case 2, let $e=e_\theta(F)=e_{\theta}(G)$ be the number of $\theta$'s appearing in the strings defining $F$ and $G$. Then, writing $F=F'\circ\theta$ as in \eqref{relation3}, we see that 
\begin{equation}\label{relation4}
\deg_a(F)=\deg_a(F')+\deg(F')\deg_a(\theta)=\deg_a(F')+\deg(F')\geq\deg(F').
\end{equation} 
On the other hand, Lemma \ref{lem:a-degree} part (2) applied to $f=G'$ implies that   
\begin{equation}\label{relation5}
\deg_a(G)=\deg_a(G')+\deg(G')\deg_a(\tau)=\deg_a(G')\leq\frac{d^e-1}{d-1}\cdot\frac{\deg(G')}{d^e}. 
\end{equation}   
Here we use that $G=G'\circ\tau$ and that $\deg_a(\tau)=0$, since $\theta\neq\tau$ and the leading coefficients of the elements in $S$ are multiplicatively independent. Therefore, if $\deg_a(F)=\deg_a(G)$, then \eqref{relation3}, \eqref{relation4}, \eqref{relation5} together imply that
\begin{equation}\label{relation6} 
\begin{split} 
\deg(\tau)\deg(F)&=\deg(\tau)d\deg(F') \\[5pt] 
&\leq\deg(\tau)d \deg_a(F) \\[5pt]
&=\deg(\tau)d \deg_a(G) \\[5pt]
&\leq \frac{d^e-1}{d-1}\cdot\frac{\deg(\tau)\deg(G')}{d^{e-1}} \\[5pt]
&\leq \frac{d^e-1}{d-1}\cdot\frac{\deg(G)}{d^{e-1}} \\[5pt]
&< \frac{d^e}{d-1}\cdot\frac{\deg(G)}{d^{e-1}} \\[5pt] 
&=\frac{d}{d-1}\deg(G).
\end{split} 
\end{equation} 
However, $F=G$ so that $\deg(F)=\deg(G)$. In particular, \eqref{relation6} implies that 
\[2\leq\deg(\tau)<\frac{d}{d-1}\leq2,\]
a contradiction. Therefore, $\deg_a(F)\neq\deg_a(G)$ and Case 2 is incompatible with \eqref{relation3}. Therefore, any relation in $M_S$ must be of the form in Case 1. However, since we have settled Theorem \ref{thm:justfree} in this case by induction, $M_S$ is a free monoid as claimed.          
\end{proof}         

We now include a proof of Lemma \ref{lem:a-degree} regarding $a$-degreess. 
\begin{proof}[(Lemma \ref{lem:a-degree})] The first statement is a simple consequence of the definition of $a$-degrees. Suppose that $f_1,f_2,g\in M_S$, that $\deg_a(f_1)\leq\deg_a(f_2)$, and that $\deg(f_1)\leq\deg(f_2)$. Then \vspace{.05cm}  
\[\deg_a(f_1\circ g)=\deg_a(f_1)+\deg_a(g)\cdot\deg(f_1)\leq\deg_a(f_2)+\deg_a(g)\cdot\deg(f_2)=\deg_a(f_2\circ g) \vspace{.05cm}  \]
as claimed. For the second statement, let $f\in M_S$ and suppose that $e_\theta(f)=e\geq1$. Then, we may write 
\begin{equation}\label{form}
f=g_{t+1}\circ \theta^{r_t}\circ g_{t}\circ\dots \circ g_2\circ \theta^{r_1}\circ g_1
\end{equation} 
for some $g_i\in M_S$ with $\deg_a(g_i)=0$, some $t\geq 1$, and some $r_i\geq0$ with $\sum_{i=1}^{t}r_i=e$. We will show by induction on $t$ that \vspace{.05cm}  
\begin{equation}\label{bd:a-degree}
\deg_a(f)\leq \deg_a(g_{e+1}\circ g_e\circ\dots\circ g_1\circ\theta^e), \vspace{.05cm}  
\end{equation} 
from which statement (2) of the Lemma easily follows. If $t=1$, then  \vspace{.05cm}  
\[\deg_a(g_2\circ \theta^{r}\circ g_1)=\deg(g_2)\deg_a(\theta^r)\leq \deg(g_2)\deg(g_1)\deg_a(\theta^r)\leq \deg_a(g_2\circ g_1\circ\theta^r). \vspace{.05cm}  \] 
Here we use that $\deg_a(g_i)=0$. On the other hand, assume that $t>1$ and that \eqref{bd:a-degree} is true for polynomials of the form in \eqref{form} with $t-1$ appearances of substrings of the form $\theta^{r_i}$. Then given $f$ as in \eqref{form}, let $f_1=g_{t+1}\circ \theta^{r_t}\circ g_{t}\circ\theta^{r_{t-1}}\circ\dots \circ g_2$, let $f_2=g_{t+1}\circ\dots \circ g_2\circ \theta^{r}$ where $r=\sum_{i=2}^{t}r_i$, and let $g=\theta^{r_1}\circ g_1$. Then $f=f_1\circ g$ and $\deg(f_1)=\deg(f_2)$. Hence, part 1 of Lemma \ref{lem:a-degree} and the induction hypothesis together imply that  \vspace{.05cm}  
\begin{equation}\label{form2}
\deg_a(f)=\deg_a(f_1\circ g)\leq \deg(f_2\circ g)=\deg_a((g_{t+1}\circ\dots\circ g_2)\circ\theta^e\circ g_1). \vspace{.05cm}  
\end{equation}
On the other hand letting $g'=g_{t+1}\circ\dots\circ g_2$, we see that the $t=1$ case above applied to $g'\circ\theta^e\circ g_1$ in place of $f$ implies that \vspace{.05cm}  
\begin{equation}\label{form3}
\deg_a((g_{t+1}\circ\dots\circ g_2)\circ\theta^e\circ g_1)\leq \deg_a((g_{t+1}\circ\dots \circ g_1)\circ\theta^e). \vspace{.05cm}    
\end{equation}     
Therefore after combining \eqref{form2} and \eqref{form3}, we establish \eqref{bd:a-degree} as claimed. Finally, the bound in part 2 of Lemma \ref{lem:a-degree} follows easily from \eqref{bd:a-degree}, the fact that \vspace{.05cm}   
\[\deg_a(\theta^e)=(d^{e-1}+\dots+d+1)=(d^e-1)/(d-1), \vspace{.05cm}  \] 
and that $\deg(g_{t+1}\circ\dots\circ g_1)=\deg(f)/d^e$.   
\end{proof} 
We are nearly ready to prove Theorems \ref{thm:rational} and \ref{thm:freemonoid}, versions of Conjecture \ref{conjecture} in dimension one, for some fairly general sets of maps. However to complete the main remaining step, (i.e., to pass from counting functions to counting points), we need to show that $f(P)=g(P)$ occurs rarely for $f,g\in M_S$ and $P$ of large enough height. This is largely achieved by ensuring that the rational (or integral) points on the curves 
\begin{equation}\label{curves}
\qquad C_i: \frac{\phi_i(x)-\phi_i(y)}{x-y}=0\qquad\text{and}\qquad C_{j,k}: \phi_j(x)=\phi_k(y)\;\;\;\text{for}\; j\neq k
\end{equation}  
are finite. For critically separate and simple sets of rational maps this follows from the genus calculations in \cite{Pakovich} and Faltings' theorem: 
\begin{proposition}{\label{prop:rational}} Let $S=\{\phi_1,\dots,\phi_s\}$ be a set of rational maps on $\mathbb{P}^1(\overline{\mathbb{Q}})$ all of degree at least $4$. If $S$ is critically separate and critically simple, then the curves in (\ref{curves}) have at most finitely many rational points over any number field.    
\end{proposition} 
\begin{proof} Since $S$ is critically separate, \cite[Proposition 3.1]{Pakovich} implies that each $C_{j,k}$ is an irreducible curve for all $j\neq k$. Likewise, it is shown on \cite[p208]{Pakovich} that the genus of $C_{j,k}$ is given by $(\deg(\phi_j)-1)(\deg(\phi_k)-1)\geq9$. Hence, the $C_{j,k}$ have at most finitely many rational points over any number field by Faltings' theorem. Likewise, \cite[Corollary 3.6]{Pakovich} implies that each $C_{i}$ is an irreducible curve. Moreover, it is shown on \cite[p210]{Pakovich} that the genus of $C_{i}$ is $(\deg(\phi_i)-2)^2\geq4$. Hence, the $C_{i}$ also have at most finitely many rational points over any number field by Faltings' theorem.      
\end{proof} 
For the sets of polynomials in Theorem \ref{thm:freemonoid}, it suffices for our purposes to show that the curves in (\ref{curves}) have only finitely many integral points (as opposed to rational points). To do this, we need the integral point classification theorems in \cite{Bilu} and the Appendix \ref{Appendix}. To put these results in context, we first recall the definition of Siegel factors and Siegel's integral point theorem. 
\begin{definition} A \emph{Siegel polynomial} over a field $K$ is an absolutely irreducible polynomial $\Phi(x,y)\in K[x,y]$ for which the curve $\Phi(x,y)=0$ has genus zero and has at most two points at infinity. A \emph{Siegel factor} of a polynomial $\Psi(x,y)\in K[x,y]$ is a factor of $\Psi$ which is a Siegel polynomial over $K$.  
\end{definition} 
The following result explains the relevance of Siegel factors in this context and is one of the most important results in arithmetic geometry; see Theorems 8.2.4 and 8.5.1 in \cite{Lang}.  
\begin{theorem}[Siegel]{\label{Siegel}} Let $R$ be a finitely generated integral domain of characteristic zero, let $K$ be the field of fractions of $R$, and let $\Phi(x,y)\in K[x,y]$. Then there are only finitely many pairs $(x,y)\in R\times R$ for which $\Phi(x,y)=0$ unless $\Phi(x,y)$ has a Siegel factor over $K$.   
\end{theorem} 
\begin{remark}{\label{rem:extension}} Clearly if $K$ is a number field (viewed inside the complex numbers) and $\Phi(x,y)$ has no Siegel factors over $\mathbb{C}$, then $\Phi(x,y)$ has no Siegel factors over $K$. Therefore, to prove that the equation $\Phi(x,y)=0$ has only finitely many solutions $(x,y)$ in some ring of $\mathcal{S}$-integers $R\subset K$, it suffices to show that $\Phi(x,y)$ has no Siegel factors over $\mathbb{C}$.     
\end{remark} 
To use Siegel's integral point theorem to show that $f(P)=g(P)$ occurs infrequently for $f,g\in M_S$ and $P$ of sufficiently large height (see Lemma \ref{injective} for a precise statement), we need the following theorem of Bilu and Tichy \cite[Theorem 10.1]{Bilu}, which classifies the polynomials $\Phi(x,y)=F(x)-G(y)$ having a Siegel factor.   
\begin{theorem}{\label{classification}} For non-constant $F,G\in\mathbb{C}[x]$, if $F(x)-G(y)$ has a Siegel Factor in $\mathbb{C}[x,y]$ then $F=E\circ F_1\circ\mu$ and $G=E\circ G_1\circ \nu$, where $E, \mu, \nu\in\mathbb{C}[x]$ with $\deg(\mu)=\deg(\nu)=1$ and either $(F_1,G_1)$ or $(G_1,F_1)$ is one of the following pairs (here $m,n\geq1$ and $p\in\mathbb{C}[x]\mysetminus\{0\}$): \vspace{.1cm} 
\begin{enumerate} 
\item[\textup{(a)}] $\big(x^m,x^rp(x)^m\big)$, where $r\in\mathbb{N}$ is coprime to m; \vspace{.15cm} 
\item[\textup{(b)}] $\big(x^2,(x^2+1)p(x)^2\big)$;  \vspace{.15cm} 
\item[\textup{(c)}] $\big(T_m,T_n\big)$ with $\gcd(m,n)=1$; \vspace{.15cm} 
\item[\textup{(d)}] $\big(T_m,-T_n\big)$ with $\gcd(m,n)>1$; \vspace{.15cm} 
\item[\textup{(e)}] $\big((x^2-1)^3,3x^4-4x^3\big)$. 
\end{enumerate}   
\end{theorem}
\begin{remark} Technically, the statement above is a simplified version of \cite[Theorem 10.1]{Bilu} taken from \cite[Corollary 2.7]{linear}. For a more detailed description of the classification of pairs $(F,G)$ such that $F(x)-G(y)$ has a Siegel factor (with the relevant fields of definition taken into account), see \cite{Bilu}.      
\end{remark} 
In particular, condition (3) of Theorem \ref{thm:freemonoid} implies that the affine curves 
\[C_{i,j}: \phi_i(x)=\phi_j(y)\;\;\;\text{for}\; i\neq j\] 
have only finitely many integral points. Here we use Theorem \ref{Siegel}, Remark \ref{rem:extension}, and Theorem \ref{classification}: the pairs (a)-(d) in Theorem \ref{classification} are ruled out by condition (3) by examining first coordinates only (all cyclic or Chebychev polynomials). Likewise, $(x^2-1)^3=F\circ E\circ L$, where $F(x)=(x-1)^3$, $E(x)=x^2$ is cyclic, and $L(x)=x$. Hence, the pair in (e) is also ruled out by condition (3). Similarly, condition (3) implies that the affine curves 
\[C_i: \frac{\phi_i(x)-\phi_i(y)}{x-y}=0\] 
have only finitely many integral points. Here we use Theorem \ref{Siegel} and Theorem \ref{T.T2} in the Appendix; Zannier has shown that such curves have at least $3$ points at infinity over $\mathbb{C}$ and thus cannot have a Siegel factor over any number field. In particular, we are now ready to prove our orbit counts for $\mathbb{P}^1$ from the Introduction.  
\begin{proof}[(Proof of Theorems \ref{thm:rational} and \ref{thm:freemonoid})] Suppose that $S$ is a critically separate and critically simple set of rational functions or that $S$ is a set of polynomials satisfying conditions (1)-(3) of Theorem \ref{thm:freemonoid}. Then in particular, $M_S$ is free by Proposition \ref{prop:rational} in the rational function case, and $M_S$ is free by Theorem \ref{thm:justfree} in the polynomial case. Hence, Theorem \ref{thm:count} implies that the number of functions $f\in M_S$ satisfying $H(f(P))\leq B$, has the desired growth rate (in either case), whenever $P$ has large enough height. 

To pass from functions to points, we need to control when $f(P)=g(P)$ is possible for $f,g\in M_S$. With this in mind, let $R_P\subset K$ be a ring of $\mathcal{S}$-integers in some number field $K$ (not the same $S$ as the set of functions) containing $P$ and the coefficients of the maps in $S$. Then define the quantities 
\[
\Scale[.93]{
\kappa_P:=\max\Big\{h(x)\,:\text{$(x,y)\in C_i(K)$ or $(x,y)\in C_{j,k}(K)$ for some $y\in K$ and some $i,j,k$}\Big\}}.\] 
in the rational function case and 
\[
\Scale[.93]{
\kappa_P:=\max\Big\{h(x)\,:\text{$(x,y)\in C_i(R_P)$ or $(x,y)\in C_{j,k}(R_P)$ for some $y\in R_P$ and some $i,j,k$}\Big\}}.\] 
in the polynomial case. Then in either case, $\kappa_P$ is finite by Proposition \ref{prop:rational}, Theorem \ref{Siegel}, Remark \ref{rem:extension}, Theorem \ref{classification}, and Theorem \ref{T.T} in the Appendix. Now given $f=\theta_1\circ\theta_2\circ\dots\circ \theta_n\in M_S$, define the length of $f$ to be $\ell(f)=n$; note that this quantity is well-defined since $M_S$ is free. Moreover letting $\textbf{v}=(1,\dots,1)$, we see that $\ell=\ell_{S,\textbf{v}}$ in our earlier notation. Next, recall the constant $b_S$ given by $b_S=C_S/(d_S-1)$, where $C_S$ and $d_S$ are the height constants in Definition \ref{def:htcontrolled} above. Then, Tate's telescoping Lemma \ref{lem:tate} implies that if $h(\rho(P))\leq \kappa_P$ for some $P$ with $h(P)>2b_S$ and some $\rho\in M_S$, then  
\begin{equation}\label{length1} 
2^{\ell(\rho)}b_S\leq\deg(\rho)(h(P)-b_S)\leq h(\rho(P))\leq\kappa_P.
\end{equation}  
Hence, the length of such $\rho$ is bounded; specifically, $\ell(\rho)\leq\max\big\{1,\lceil\log_2(\kappa_P/b_S)\rceil\big\}:=r_P$, from which we deduce the following fact.    
\begin{lemma}\label{lem:length1} Suppose that $S$ satisfies the conditions of Theorems \ref{thm:rational} or \ref{thm:freemonoid} and let $\rho\in M_S$. If $\ell(\rho)>r_P$, $h(P)>2b_S$, and $\theta(\rho(P))=\tau(P')$ for some $P'\in R_P$ and some $\theta,\tau\in S$, then $\theta=\tau$ and $\rho(P)=P'$.  
\end{lemma}
In particular, this allows us to control the number of functions in $M_S$ that can agree at $P$. 
\begin{lemma}\label{injective} Suppose that $S$ satisfies the conditions of Theorems \ref{thm:rational} or \ref{thm:freemonoid} and $h(P)>2b_S$. Then there is a constant $t_{P,S}$ depending only on $P$ and $S$ such that 
\[\#\big\{f\in M_S\,:\, f(P)=Q\big\}\leq t_{P,S}\] 
holds for all but finitely many $Q\in\Orb_S(P)$. 
\end{lemma} 
\begin{proof} Let $\mathfrak{d}_S=\max\{\deg(\phi):\phi\in S\}$ and suppose that $Q\in\Orb_S(P)$ satisfies  
\[h(Q)>{\mathfrak{d}_S}^{r_P+1}(h(P)+b_S),\] true of all but finitely many $Q$ by Northcott's Theorem;  each $Q\in\Orb_S(P)\subseteq\mathbb{P}^1(K)$ by construction of $K$. Then, it follows from Tate's telescoping Lemma \ref{lem:tate} that $\ell(f)>r_P+1$ for all $f\in M_S$ with $f(P)=Q$: otherwise, 
\[h(Q)=h(f(P))\leq\deg(f)(h(P)+b_S)\leq{\mathfrak{d}_S}^{\ell(f)}(h(P)+b_S)\leq{\mathfrak{d}_S}^{r_P+1}(h(P)+b_S),\]
a contradiction. In particular, each function taking the value of $Q$ at $P$ has length strictly larger than $r_P$+1. Now, let $f_Q\in M_S$ be a function of smallest length taking the value of $Q$ at $P$. Then $\ell(f_Q)>r_p+1$ and we may write $f_Q=\tau_1\circ\dots\circ\tau_m\circ\rho_Q$ form some $\tau_i\in S$, some $m\geq1$,  and some $\rho_Q\in M_S$ of length $r_P+1$. Likewise, for any other $f\in M_S$ with $f(P)=Q$, we may write $f=\theta_{1,f}\circ\dots\circ\theta_{m,f}\circ q_f\circ\rho_f$ for some $\theta_{i,f}\in S$, some $q_f\in M_S$, and some $\rho_f\in M_S$ of length $r_P+1$; here we use the minimality of the length of $f_Q$. Then $f(P)=f_Q(P)$ implies: 
\begin{equation}\label{length1}
\theta_{1,f}\circ\dots\circ\theta_{m,f}\circ q_f\circ\rho_f(P)=\tau_1\circ\dots\circ\tau_m\circ\rho_Q(P). 
\end{equation}
Now for all $1\leq i\leq m$, let $\rho_i=\theta_{i+1,f}\circ\dots\circ\theta_{m,f}\circ q_f\circ\rho_f$ and $P_i'=\tau_{i+1}\circ\dots\circ\tau_m\circ\rho_Q(P)$. In particular, \eqref{length1} becomes  
\[\theta_{1,f}(\rho_1(P))=\tau_1(P_1').\]
On the other hand, $P_i'\in R_P$ by definition of $R_P$ and $\ell(\rho_i)\geq\ell(\rho_{f})=r_P+1>r_P$ for all $i$. Hence, Lemma \ref{lem:length1} applied to $\rho=\rho_1$, $P'=P_1'$, $\theta=\theta_{1,f}$, and $\tau=\tau_1$ implies that $\theta_{1,f}=\tau_1$ and $\rho_1(P)=P_1'$. Therefore, 
\[\theta_{2,f}\circ\dots\circ\theta_{m,f}\circ q_f\circ\rho_f(P)=\tau_2\circ\dots\circ\tau_m\circ\rho_Q(P). \] 
Repeating the same argument, this time with $\rho=\rho_2$, $P'=P_2'$, etc., we see that Lemma \ref{lem:length1} implies that $\theta_{2,f}=\tau_2$ and  $\rho_2(P)=P_2'$. We can clearly continue this argument ($m$-times) and obtain that 
\begin{equation}\label{length2}
q_f\circ\rho_f(P)=\rho_Q(P)\qquad\text{and}\qquad\text{$\theta_{i,f}=\tau_i$ \;for all $1\leq i\leq m$}.
\end{equation} 
On the other hand, Tate's Telescoping Lemma \ref{lem:tate} and the fact that $h(P)>2b_S$ imply the lower bound 
\begin{equation}\label{length3}
2^{r_P+1}b_S\leq\deg(\rho_f)(h(P)-b_s)\leq h(\rho_f(P)). 
\end{equation} 
Likewise, we have the upper bound
\begin{equation}\label{length4}
h(\rho_Q(P))\leq\deg(\rho_Q)(h(P)+b_S)\leq{\mathfrak{d}_S}^{r_P+1}(h(P)+b_S).
\end{equation} 
Hence, after combining \eqref{length2}, \eqref{length3} and \eqref{length4} with Lemma \ref{lem:tate} applied to the map $q_f$, we see that \vspace{.05cm} 
\[
\Scale[.9]{\deg(q_f)(2^{r_P+1}-1)b_S\leq\deg(q_f)(h(\rho_f(P))-b_S)\leq h(q_f\circ\rho_f(P))=h(\rho_Q(P))\leq{\mathfrak{d}_S}^{r_P+1}(h(P)+b_S)}.  \vspace{.15cm}\] 
In particular, dividing both sides of the inequality above by $(2^{r_P+1}-1)b_S$, we deduce that 
\begin{equation}\label{length5}
2^{\ell(q_f)}\leq\deg(q_f)\leq\frac{{\mathfrak{d}_S}^{r_P+1}(h(P)+b_S)}{(2^{r_P+1}-1)b_S}. 
\end{equation}
Hence the length of $q_f$ is bounded. But $S$ is a finite set of maps, so the number of possible $q_f$'s is finite. Likewise, the length of $\rho_f$ is $r_P+1$ is bounded, and so there are only finitely many possible $\rho_f$'s. In summation, we have shown that if $f\in M_S$ is \emph{any} function with $f(P)=Q$, then $f=\tau_1\circ\dots\circ\tau_m\circ q_f\circ\rho_f$ such that: the $\tau_i$ are fixed, and the number of possible $q_f$'s and $\rho_f$'s are bounded independently of $Q$. Specifically, we have that
\[\#\big\{f\in M_S\,:\, f(P)=Q\big\}\leq s^{\log_2\Big\lceil\frac{{\mathfrak{d}_S}^{r_P+1}(h(P)+b_S)}{(2^{r_P+1}-1)b_S}\Big\rceil+r_P+1}\vspace{.1cm}\]
holds for all $Q\in\Orb_S(P)$ with $h(Q)>{\mathfrak{d}_S}^{r_P+1}(h(P)+b_S)$, which proves the claim.                          
\end{proof} 

We now finish the proof of Theorems \ref{thm:rational} and \ref{thm:freemonoid}. Note that Lemma \ref{injective} implies that: 
\[
\Scale[.85]{
t_{P,S}^{-1}\cdot\#\big\{f\in M_S\,:\, H(f(P))\leq B\big\}+O(1)\leq\#\big\{Q\in\Orb_S(P)\,:\, H(Q)\leq B\big\}\leq\#\big\{f\in M_S\,:\, H(f(P))\leq B\big\}} \vspace{.1cm}\]
holds for all $B$ sufficiently large and all $P$ such that $H(P)>e^{2b_S}$. Moreover, combining the  bounds above with Theorem \ref{thm:count}, we see that for all $\epsilon>0$ there exists an effectively computable positive constant $b=b(S,\epsilon)$ such that \vspace{.05cm} 
\begin{equation*}
(\log B)^{b}\ll\#\{Q\in \Orb_S(P)\,:\, H(Q)\leq B\}\ll (\log B)^{b+\epsilon}
\end{equation*}
as desired.   
\end{proof}
In higher dimensions, it is possible that one can attack Conjecture \ref{conjecture} in a similar manner to that above, provided that one can give a reasonable condition ensuring that the set of rational/integral points on the variety  
\[V_{f,g}:=\{(P,Q)\in\mathbb{P}^N\times\mathbb{P}^N\,:\, f(P)=g(Q)\}\]
is not Zariski dense (for all distinct $f,g\in M_S$ of some fixed length). To do this, it is likely necessary to assume the Bombieri-Lang Conjecture.

Likewise (although most sets generate free monoids), it would be interesting to study the height growth rates in monoid orbits which are \emph{not} free (or free commutative). As a test case, one might consider the following example from  \cite[Remark 1.5]{Zieve}: let $\omega$ be a primitive cube root of unity and let $F(x)=x^2$ and $G(x)=\omega x^2$. Then the monoid generated by $S=\{F,G\}$ has three independent relations: $F^2=G^2$, $F^2\circ G=G\circ F^2$, and $G\circ F\circ G=F\circ G\circ F$. \\[2pt]

\section{Appendix: integral points on curves ${f(X)-f(Y)\over X-Y}$ \\[2pt] (by Umberto Zannier)}\label{Appendix} 
\vspace{.1cm}
Let $f\in \C[X]$ be a  polynomial of degree $d\ge 2$ and let $\mathcal{O}$ be a finitely generated subring of $\C$. For  the sequel we  put
\begin{equation}\label{F}
F(X,Y)={f(X)-f(Y)\over X-Y}.
\end{equation}
Recall also that the {\it cyclic polynomial} of degree $n$ is simply $X^n$, and the {\it Chebyshev polynomial} of degree $n$  is the unique polynomial  $T_n$ satisfying the identity $T_n(Z+Z^{-1})=Z^n+Z^{-n}$. The purpose of the present Appendix is to prove the following: \vspace{.05cm} 

\begin{thm} \label{T.T} Assume that the plane curve defined by $F(X,Y)$ has infinitely many points in $\mathcal{O}^2$. Then there are an integer $n>1$ and polynomials $g, l \in\C[X]$, with $\deg l=1$, such that $f=g\circ S_n\circ l$, where  $S_n$ is either the cyclic or the Chebyshev polynomial of degree $n$. \vspace{.05cm} 
\end{thm}

\begin{rem} Note that the result has an easy converse, as soon as we allow some freedom on $\mathcal{O}$, as we now illustrate: \vspace{.2cm} 

(i) If $S_n(X)=X^n$  (after applying $l^{-1}$)   we obtain factors $X-\zeta Y$  ($\zeta^n=1$, $\zeta\neq 1$)  for our polynomial $F(X,Y)$, i.e. components of the curve which are lines defined over $\Q(\zeta)$. Therefore we obtain infinitely many points in $\mathcal{O}^2$ as soon as $\mathcal{O}$ contains $\zeta$ (and the coefficients of $l$).  \vspace{.2cm} 

(ii)  In the case $S_n=T_n$, from the defining property of $T_n$ we easily obtain (well-known) factors of $T_n(X)-T_n(Y)$ given by 
$X^2-(\zeta+\zeta^{-1})XY+Y^2+(\zeta-\zeta^{-1})^2$, for $\zeta\neq \pm 1$ an $n$-th root of unity. On setting $Y=W+W^{-1}$, this quadratic  in turn factors as $(X-\zeta W-\zeta^{-1}W^{-1})(X-\zeta^{-1}W-\zeta W^{-1})$.  Hence, if we let $w$ take values in $\mathcal{O}^*$ (which may well be infinite) and set $X=x=\zeta w+\zeta^{-1}w^{-1}$ we obtain again an infinity of points in $\mathcal{O}^2$.  We also obtain similarly  quadratic factors of $T_n(X)+T_n(Y)$, which are relevant when  $g(X)=h(X^2)$ is even. These factors divide also $T_{2n}(X)-T_{2n}(Y)$, since $T_{2n}=T_2\circ T_n=T_n^2-2$.

In the next version of the result, i.e. Theorem \ref{T.T2} below, we shall add a further conclusion which implies that all but finitely many integral points arise in this way. 
\end{rem}
\vspace{.05cm} 

As to the theorem, we recall at once that in virtue of  Siegel's Theorem   (extended suitably to finitely generated subrings) an {\it  irreducible}  affine  curve can have can have infinitely many (integral) points defined over $\mathcal{O}$ only if

 (i) it has genus $0$\qquad  and
 
  (ii) it has at most two points at infinity. \footnote{By {\it points at infinity} we mean the missing points with respect to a projective closure of the curve. This number may increase by passing to a smooth model, but the theorem applies to any model.}   
  
  See \cite{BG}, or \cite{L}, or \cite{SeMWT}. The crucial case is the original Siegel's 1929 version  over $\Z$, as extended later by Mahler to the rings of $S$-integers in  a number field. 
  
\medskip

Thus the problem is to investigate when the (possibly reducible) curve defined by $F(X,Y) $  has a component satisfying these `Siegel conditions' (which   cannot generally be improved).

This leads  in the first place  to the need to establish   {\it when the defining polynomial $F$ can be reducible}. If $f$ is {\it indecomposable} (i.e. not of the shape $g\circ h$ for polynomials $g,h$ of degree $>1$) then the  correct condition was found by Fried \cite{F}: namely, {\it $F$ is irreducible  unless $f(X)$ is either a cyclic or a Chebyshev polynomial up to a linear change of variable}, which of course corresponds to our conclusion. (See also Schinzel's book \cite{Sch}, especially 1.5, where  fields of definitions are considered as well, which instead we disregard here.)  An application of Fried's result would then directly yield   the present theorem  in the indecomposable cases. 

However,  if   $f$ is decomposable then  certainly $F(X,Y)$ is anyway reducible, and the issue leads to more delicate problems concerning the nature of the irreducible factors.  In the paper \cite{AZ} a  laborious classification is obtained for all the cases when there is a factor defining a curve of genus $0$.  The results of \cite{AZ} depend on some finite-group theory, which is used to an even much heavier extent in Mueller's paper \cite{M}, which again obtains  certain complete laborious classifications  relevant for suitable  applications of Siegel's theorem.

An applications of  \cite{AZ} would suffice for the present purposes of proving Theorem \ref{T.T}, even forgetting about Siegel's condition (ii). But in fact  it turns out that adding such  condition not only makes the former (i) automatic, but also leads to a much simpler  and self-contained elementary proof, which can be hopefully useful for some readers and for other applications.  Moreover this proof yields with little effort a slightly more precise conclusion, as in the last phrase of the statement below (which, as in the Remark above, allows to describe all but finitely many integral points). 

To present such a proof is the scope of this Appendix.
By the remarks above,  for Theorem \ref{T.T} it will suffice  to prove the following result (even disregarding the last conclusion):  \vspace{.2cm}  

\begin{thm} \label{T.T2} Assume that the polynomial  $F(X,Y)$ has an irreducible factor  $\Phi$ defining a curve 
with at most two points at infinity (in a closure in $\P_2$). Then $\deg\Phi\le 2$ and  there are an integer $n>1$ and polynomials $g, l \in\C[X]$, with $\deg l=1$, such that $f=g\circ S_n\circ l$, where  $S_n$ is  the  cyclic (if $\deg\Phi=1$)  or the Chebyshev   (if $\deg\Phi=2$)  polynomial of degree $n$.

If $\deg \Phi=1$, then $\Phi$ divides $l(X)^n-l(Y)^n$.  If $\deg\Phi=2$,  then    $\Phi$ is symmetric and either  it divides  $S_n(l(X))-S_n(l(Y))$, or $g$ is even  and  $\Phi$ divides 
$S_n(l(X))+S_n(l(Y))$. \vspace{.2cm}  
\end{thm}

\begin{proof} To start with, we normalize $f$ by assuming it is monic and with vanishing second coefficient: $f(X)=X^d+f_2X^{d-2}+\ldots +f_d$, $f_i\in\C$. This does not affect the results on taking into account the linear polynomial $l(X)$ in the statement. 

Our affine (possibly reducible)  curve $C_F:F(X,Y)=0$ has degree $d-1$. Note that the points at infinity in $\P_2$ of  (the closure of) this curve are given in homogenous coordinates $(x:y:z)$ by $z=0$, $x^d=y^d$, $x\neq y$, so they form a set of $d-1$ pairwise distinct points.\footnote{They are smooth points, which simplifies things as we do not need to refer to smooth models.} 

\medskip

Let $\Phi(X,Y)\in\C[X,Y]$ be an irreducible factor of $F(X,Y)$, defining an irreducible curve $C_\Phi$ with at most two points at infinity. The homogeneous part of $\Phi$ of highest degree must be a factor of $(X^d-Y^d)/(X-Y)$,  and  the points at infinity correspond to linear factors of this homogeneous part. Since this has not multiple factors, we deduce that $C_\Phi$ has $\deg \Phi$ points at infinity. Hence, if $C_\Phi$ satisfies Siegel's condition (ii), we must have  $\deg\Phi\le 2$.

From these considerations it also follows that we may  assume that $\Phi$ is monic in $Y$.

\medskip

Suppose first that $\deg\Phi=1$, so $\Phi(X,Y)=Y-aX-b$; hence   we must have $f(aX+b)=f(X)$ identically. Since however $f$ has vanishing second coefficient, this entails $b=0$, hence $f(aX)=f(X)$. We already know that $a$ is a $d$-th root of unity, $a\neq 1$. If $n$ is the exact order of $a$, then $n>1$ divides $d$ and $f$ must be a polynomial in $X^n$, i.e. $f(X)=g(X^n)$  and we fall into one of the cases of the conclusion.

Note that $Y-aX$ divides indeed $X^n-Y^n$ so the last assertion holds as well.

\bigskip

Suppose now that $\deg\Phi=2$. The two points at infinity of $C_\Phi$ correspond to two Puiseux expansions $Y=P_\pm(X):=a_{\pm}X+b_{0\pm}+b_{1\pm}X^{-1}+\ldots $ in descending powers of $X$, where $b_{i\pm}$ are complex numbers and $a_{\pm}$ are two  distinct $d$-th roots of $1$, both different from $1$. 

We have $\Phi(X,P_{\pm}(X))=0$ hence $F(X,P_{\pm}(X))=0$, so 
$f(X)=f(P_{\pm}(X))$ identically. As before, since $f$ has vanishing second coefficient this yields $b_{0\pm}=0$.  We may write
\begin{equation*}
\Phi(X,Y)=(Y-a_+X)(Y-a_-X)+L(X,Y)-k,
\end{equation*}
where $L$ is linear homogeneous and $k\in\C$. We have that $P_{\pm}(X)-a_{\pm}X=O(X^{-1})$, in the sense that it is a Puiseux series where no non-negative power of $X$ appears. Since $\Phi(X,P_{\pm}(X))=0$ we get that $L(X,P_{\pm}(X))=O(1)$ for both choices of the sign. But then, since $a_\pm$ are distinct this implies $L=0$, and since $\Phi$ is irreducible we have $k\neq 0$. Hence, setting $s:=a_++a_-$, $p:=a_+a_-$, we have $pk\neq 0$ and 
\begin{equation*}\label{Phi}
\Phi(X,Y)=(Y-a_+X)(Y-a_-X)-k=Y^2-sXY+pX^2-k.
\end{equation*} 

\medskip

Let now $x$ be a variable over $\C$ and let $y$ be a solution of $\Phi(x,y)=0$ in an extension of $\C(x)$, so $\F:=\C(x,y)$ is the function field of $C_\Phi$. Note that $\F$ is a  quadratic extension of both $\C(x)$ and $\C(y)$; looking at the equation we find that the  Galois groups are  generated respectively by the automorphisms $\sigma,\tau$ of $\F$ (of order $2$) given  by 
\begin{equation*}\label{Gal}
\sigma(x)=x, \quad \sigma(y)=sx-y\qquad  \tau(x)=\big({s\over p}\big)y-x,\quad \tau(y)=y.
\end{equation*}

It will be notationally convenient to have another expression for $\F$. Define the linear forms $Z_\pm :=Y-a_\pm X$, so $\Phi=Z_+Z_--k$. Letting $z_\pm=y-a_\pm x$ we thus have $z_+z_-=k$ and
\begin{equation*}\label{z}
x={z_+-z_-\over a_--a_+}=\gamma(z_+-z_-),\qquad  y=\gamma(a_-z_+-a_+z_-),
\end{equation*}
where we have put $\gamma:=(a_--a_+)^{-1}$.   So in particular we have $\F=\C(z_+)$ and by an easy computation one finds that  the above automorphisms are expressed by 
\begin{equation}\label{Gal2}
\sigma(z_+)=-z_-={\alpha\over z_+}, \qquad  \tau(z_+)= -{a_+\over a_-}z_-={\beta\over z_+},
\end{equation}
where $\alpha=-k$, $\beta=-ka_+/a_-$.

\medskip

Now, since $\Phi(x,y)=0$ we have $F(x,y)=0$ whence $f(x)=f(y)$, so the field $K:=\C(x)\cap\C(y)$  contains $\C(f(x))$  and thus the degree $[\F:K]$ is finite. The  field $K$ is left fixed by both $\sigma,\tau$, and thus by the group $G$  that they generate inside Aut$(\F/\C)={\rm PGL}_2(\C)$.  By basic Galois theory actually the fixed field of $G$ is precisely the intersection $\C(x)\cap\C(y)=K$. 

\medskip

We have $\sigma(\tau(z_+))=(\beta/\alpha)z_+$, hence $\beta/\alpha=a_+/a_-$ is a root of unity  of a certain order $n$: actually, we already knew that $a_+,a_-$ are $d$-th roots of unity, and they are distinct, so  $n>1$ is a divisor of $d$. 

The group $G$ is generated by $\sigma$ and $\xi:=\sigma\tau$. On looking at the action on $z_+$ it is now   easily seen that  $\sigma^{-1} \xi\sigma=\xi^{-1}$, so $G$  is a dihedral group of order $2n$.

\medskip

Now, the rational  function of $z_+$ given by $w:=z_+^n+\alpha^nz_+^{-n}$ of degree $2n$ is plainly invariant by both $\sigma$ and $\xi$, hence by $G$. Again by simple Galois theory, we have $\C(w)=K$. Therefore $f(x)$, which lies in $K$,  is a rational function of $w$, $f(x)=g(w)$.
(On comparing degrees we find $\deg g=d/n$.)

\medskip

Recall that $x=\gamma(z_+-z_-)=\gamma (z_++(-k)z_+^{-1})=\gamma(z_++\alpha z_+^{-1})$. Hence $x$ has only the poles $z_+=0,\infty$, and the same holds for $f(x)$ (as functions of $z_+$). It follows at once that $g$ must be  a polynomial, of degree $d/n$. 

\medskip

The proof is now easily completed by a simple change of variables.  We have $w\in K\subset \C(x)$, so we may write $w=S(x)$ with $S$ a rational function of degree $n$, which as above must be a polynomial.   

Set $z=\delta z_+$ where $\delta^2\alpha=1$. Hence $x=\gamma\delta^{-1}(z+z^{-1})$.  Also, $w=\delta^{-n}(z^n+z^{-n})$. Hence $\delta^nS(\gamma\delta^{-1}(z+z^{-1}))=z^n+z^{-n}$, and by uniqueness it follows that $\delta^nS(\gamma\delta^{-1} X)=T_n(X)$ is the Chebyshev polynomial of degree $n$. Hence in conclusion we find
\begin{equation*}
f(X)=g(\delta^{-n}T_n(\gamma^{-1}\delta X)),
\end{equation*}
as required. 

\medskip

To check the last assertion, for notational simplification we  slightly change conventions and replace $g(\delta^{-n}X)$ with $g(X)$ and $f(X)$ with $f(\gamma\delta^{-1}X)$, so to suppose $f(X)=g(T_n(X))$.  In the above notation, $x$ becomes $z+z^{-1}$ and $y=a_-z+a_+z^{-1}$. (Note that  these  substitutions leave unchanged the set $\{a_+,a_-\}$.)

Also,  let $\mu^2=a_+/a_-$, so $\mu^n=:\epsilon\in\{\pm 1\}$. We have 
$y=\mu a_-((z/\mu)+(z/\mu)^{-1})$, so  $T_n((\mu a_-)^{-1}y)=\epsilon (z^n+z^{-n})=\epsilon T_n(x)$. Hence, setting $\nu:=(\mu a_-)^{-1}$, we have 
\begin{equation*}
T_n(\nu y)=\epsilon T_n(x),\qquad g(T_n(y))=f(y)=f(x)=g(T_n(x))=g(\epsilon T_n(\nu y)).
\end{equation*}

Denoting $b:=\deg g=d/n$, we then deduce that
$\deg (T_n(y)^b-(\epsilon T_n(\nu y))^b\le (b-1)n$. But on factoring the left side  and noting that all factors but at most one have degree $\ge n$, this implies that in fact  one of the factors is constant, hence \footnote{This argument is fairly standard.}
\begin{equation}\label{nu}
T_n(y)=\theta \epsilon T_n(\nu y) +c,\qquad g(\theta X+c)=g(X),
\end{equation}
for some $b$-th root of unity $\theta$. Note  that all of these equalities  hold identically. 

Now,  the Chebyshev polynomial $T_n(X)$ starts with $X^n-nX^{n-2}+\ldots $, whence the first of the equations gives $\theta\epsilon\nu^n =\nu^2=1$. Also, if $n$ is odd then $T_n(0)=0$ whence $c=0$; if  $n$ is even then $\nu^n=1$ so $\theta\epsilon=1$ and again setting $y=0$ we find $c=0$ anyway. Conversely, if these equalities hold  it is easy to check that the equation holds, since $T_n$ has the same parity of $n$. So we may suppose in the sequel that $\theta\epsilon\nu^n =\nu^2=1$ and that $g(\theta X)=g(X)$.  

\medskip

Now, consider again the equation $T_n(\nu y)=\epsilon T_n(x)$, i.e. $\nu^nT_n(y)=\epsilon T_n(x)$. 

If $\nu^n=\epsilon$ we have $T_n(x)=T_n(y)$  so 
$\Phi(X,Y)$ divides $(T_n(X)-T_n(Y))/(X-Y)$, and we are in the first case of the conclusion. 

If $\nu^n\neq \epsilon$, then $T_n(x)=-T_n(y)$ hence $\Phi(X,Y)$ divides $T_n(X)+T_n(Y)$. Also,  we have already observed that   $\theta=\epsilon\nu^n$ which in this case equals $-1$ so $g$ is an even polynomial by the second equation in \eqref{nu} (since $c=0$), again as in the sough conclusion.

\medskip

Finally, from the above equations we derive 
\[p=a_+a_-=(a_+/a_-)(a_-)^2=(\mu a_-)^2=\nu^{-2}=1,\] 
hence $\Phi(X,Y)$ is symmetric.

\medskip

This concludes the proof of Theorem \ref{T.T2}. 
\end{proof}

\begin{rem} Actually, the proof  yields some small supplementary information on the structure of the factors (which however can be deduced independently {\it a posteriori}). 

We also note that the last conclusion could have been stated as follows: {\it if $n$ is maximal such  that the   decomposition  holds, then the quadratic factor anyway divides $S_n(l(X))-S_n(l(Y))$}. 

Indeed, if $g$ is even, then since $T_2(X)=X^2-2$, $g$  can be written as $h\circ T_2$ and now we use that $T_2\circ T_n=T_{2n}$ (well known and easy to deduce). 
\end{rem}

 \vspace{1cm}

Wade Hindes 

\noindent
Texas State University\\ 
601 University Dr.\\ 
San Marcos, TX 78666. \\
\email{wmh33@txstate.edu}
 
 \vspace{1cm}

Umberto Zannier

\noindent
Scuola Normale Superiore\\
Piazza dei Cavalieri, 7\\
56126 PISA -- Italy\\
umberto.zannier@sns.it

\end{document}